\documentclass[12pt]{amsart}

\usepackage[pdfauthor   = {Mohamed\ Barakat\ and\ Markus\ Lange-Hegermann},
            pdftitle    = {A\ constructive\ approach\ to\ the\ module\ of\ twisted\ global\ sections\ on\ relative\ projective\ spaces},
            pdfsubject  = {},
            pdfkeywords = {Serre\ quotient\ category;\ reflective\ localization\ of\ Abelian\ categories;\ Gabriel\ monad;\ (truncated)\ module\ of\ twisted\ global\ sections;\ direct\ image\ functor;\ linear regularity;\ Gröbner\ bases;\ saturation},
            bookmarks=true,
            bookmarksopen=true,
            pagebackref=true,
            hyperindex=true,
            colorlinks=true,
            linkcolor=blue,
            citecolor=blue,
            filecolor=blue,
            urlcolor=blue,
            ]{hyperref}

\usepackage[utf8,utf8x]{inputenc}
\usepackage[english]{babel}
\usepackage[T1]{fontenc}
\usepackage{geometry}                
\usepackage{times}
\usepackage{mathrsfs}
\usepackage{latexsym}
\usepackage{amssymb}
\usepackage{amsthm}
\usepackage{epsfig}
\usepackage{colortbl}
\usepackage[all]{xy}
\usepackage{fancyvrb}
\usepackage{graphicx}
\usepackage[dvipsnames]{xcolor}
\usepackage{accents} 
\usepackage{enumerate}

\usepackage{wrapfig}
\usepackage{tikz}
\usetikzlibrary{shapes,arrows,matrix,backgrounds,positioning,plotmarks,calc,patterns,
decorations.shapes,
decorations.fractals,
decorations.markings,
decorations.pathreplacing,
decorations.pathmorphing,
decorations.text
}
\usepackage[colorinlistoftodos,shadow]{todonotes}

\usepackage{multirow}
\usepackage{mdwlist}
\usepackage{stmaryrd}
\usepackage{mathdots} 

\usepackage[toc,page]{appendix}
\usepackage{float}

\usepackage[sort&compress,capitalise]{cleveref}

\newtheoremstyle{mytheoremstyle} 
    {5pt}                    
    {5pt}                    
    {\itshape}                   
    {\parindent}                           
    {\bf}                   
    {.}                          
    {.5em}                       
    {}  

\theoremstyle{mytheoremstyle}

\newtheorem{theorem}{Theorem}[section]

\newtheorem{lemm}[theorem]{Lemma}
\newtheorem{prop}[theorem]{Proposition}
\newtheorem{coro}[theorem]{Corollary}

\newtheoremstyle{mytdefintionstyle} 
    {5pt}                    
    {5pt}                    
    {\rm}                   
    {\parindent}                           
    {\bf}                   
    {.}                          
    {.5em}                       
    {}  

\theoremstyle{remark}
\newtheorem{rmrk}[theorem]{Remark}

\theoremstyle{mytdefintionstyle}
\newtheorem{defn}[theorem]{Definition}
\newtheorem{exmp}[theorem]{Example}
\newtheorem{algo}[theorem]{Algorithm}

\newtheoremstyle{exmp_contd}
    {5pt}                    
    {5pt}                    
    {\rm}                   
    {\parindent}                           
    {\bf}                   
    {.}                          
    {.5em}                       
    {\thmname{#1}\ \thmnumber{ #2}\thmnote{#3}\ (continued)}  
\theoremstyle{exmp_contd}
\newtheorem*{exmp_contd}{Example}

\newcommand\nameft\textrm

\newcommand{\QQ}{{\mathscr{Q}}}
\renewcommand{\SS}{{\mathscr{S}}}

\newcommand{\WW}{{\mathscr{W}}}

\DeclareMathOperator{\End}{End}
\DeclareMathOperator{\Ext}{Ext}

\DeclareMathOperator{\coker}{coker}

\DeclareMathOperator{\Hom}{Hom}

\DeclareMathOperator{\GradedHom}{Hom_\bullet}

\newcommand{\GradedExt}{\operatorname{Ext}_\bullet}
\DeclareMathOperator{\Spec}{Spec}

\DeclareMathOperator{\Sh}{Sh}

\DeclareMathOperator{\Tor}{Tor}

\newcommand{\bS}{\mathbf{S}}
\newcommand{\bSd}{\mathbf{S}^{\geq d}}
\DeclareMathOperator{\Sat}{Sat}

\newcommand{\Egrlin}{{E\mathrm{\textnormal{-}}\operatorname{grlin}}}
\newcommand{\Egrlind}{{E\mathrm{\textnormal{-}}\operatorname{grlin}^{\ge d}}}
\newcommand{\Egrlinzero}{E\mathrm{\textnormal{-}}\operatorname{grlin}^0}
\newcommand{\Egrlindzero}{E\mathrm{\textnormal{-}}\operatorname{grlin}^{\ge d,0}}

\newcommand{\Sgrmod}{{\grS\mathrm{\textnormal{-}grMod}}}
\newcommand{\Sfpmod}{{\grS\mathrm{\textnormal{-}mod}}}
\newcommand{\Afpgrmod}{{A\mathrm{\textnormal{-}grmod}}}

\newcommand{\Sfpgrmod}{{\grS\mathrm{\textnormal{-}grmod}}}
\newcommand{\SfpgrMod}{{\grS\mathrm{\textnormal{-}grMod}}}
\newcommand{\Efpgrmod}{{E\mathrm{\textnormal{-}grmod}}}
\newcommand{\Sfpgrmodd}{\grS\mathrm{\textnormal{-}grmod}_{\geq d}}
\newcommand{\Sqfgrmod}{{\grS\mathrm{\textnormal{-}qfgrmod}}}

\newcommand{\Coh}{\mathfrak{Coh}\,}

\newcommand\grS{S}

\newcommand\grM{M}

\newcommand\shF{\mathcal{F}}

\newcommand\m{\mathfrak{m}}

\renewcommand\O{\mathcal{O}}
\newcommand\PP{\mathbb{P}}

\newcommand\A{\mathcal{A}}
\newcommand\C{\mathcal{C}}

\newcommand\F{\mathcal{F}}

\renewcommand{\O}{\mathcal{O}}

\DeclareMathOperator\linreg{linreg}
\newcommand\linregd{\linreg_{\geq d}}
\DeclareMathOperator\reg{reg}

\newcommand\T{\mathbf{T}}

\newcommand{\Z}{\mathbb{Z}}

\renewcommand\phi{\varphi}

\DeclareMathOperator\Sym{Sym}
\DeclareMathOperator\id{id}
\DeclareMathOperator\Id{Id}
\newcommand{\RR}{\mathbf{R}}
\newcommand{\RRd}{\mathbf{R}^{\ge d}}

\newcommand{\MM}{\mathbf{M}}
\newcommand{\MMd}{\mathbf{M}_{\ge d}}

\DeclareMathOperator\cores{co-res}

\definecolor{darkgray}{rgb}{0.3,0.3,0.3}

\pgfarrowsdeclarecombine[\pgflinewidth]
  {doublestealth}{doublestealth}{stealth'}{stealth'}{stealth'}{stealth'}

\definecolor{darkgreen}{rgb}{0.008,0.617,0.067}
\definecolor{brown}{rgb}{0.6,0.4,0.2}

\newif\ifjournalversion

\author{Mohamed Barakat}
\address{Department of mathematics, University of Siegen, 57068 Siegen, Germany}
\email{\href{mailto:Mohamed Barakat <mohamed.barakat@uni-siegen.de>}{mohamed.barakat@uni-siegen.de}}

\author{Markus Lange-Hegermann}
\address{}
\email{\href{mailto:Markus Lange-Hegermann <markus.lange.hegermann@rwth-aachen.de>}{markus.lange.hegermann@rwth-aachen.de}}

\begin{document}

\title[A constructive approach to the module of twisted global sections]{A constructive approach to the module of twisted global sections on relative projective spaces}

\begin{abstract}
  The ideal transform of a graded module $M$ is known to compute the module of twisted global sections of the sheafification of $M$ over a relative projective space.
  We introduce a second description motivated by the relative BGG-correspondence.
  However, our approach avoids the full BGG-correspondence by replacing the \nameft{Tate} resolution with the computationally more efficient purely linear saturation and the \nameft{Castelnuovo-Mumford} regularity with the often enough much smaller linear regularity.
  This paper provides elementary, constructive, and unified proofs that these two descriptions compute the (truncated) modules of twisted global sections.
  The main argument relies on an established characterization of \nameft{Gabriel} monads.
\end{abstract}


\maketitle


\section{Introduction}

We consider coherent sheaves over the projective space $\PP_B^n$ for a suitable ring $B$.
Any such coherent sheaf $\F$ can be described by a graded module over the polynomial ring $S:=B[x_0,\ldots,x_n]$.
Even though this representation is not unique, among the different graded $S$-modules representing $\F$ there is the distinguished representative $H^0_\bullet(\F)$, the module of twisted global sections\footnote{$H^0_\bullet(\F) := \bigoplus_{p\in \Z}H^0(\PP_B^n,\F(p))$.}.
In general the module of twisted global sections is not finitely generated, but any of its truncations $H^0_{\geq d}(\F)$ is.

In this paper we treat the functor $H^0_{\geq d}$ constructively.
It is well-known that the ideal transform functor computes $H^0_{\geq d}$, which we state as Theorem \ref{thm:Dm}.
Furthermore, in Theorem \ref{thm:tate_monad} we present a new recursive algorithm, inspired by \cite{EFS,ES}, to compute $H^0_{\geq d}$.
The \nameft{Tate} resolution in loc.\  cit.\  incorporates all higher cohomologies\footnote{See the introduction to Section~\ref{sec:saturate_linear_complex} and Remark~\ref{rmrk:E1}.}, whereas our new algorithm introduces a smaller complex, called the purely linear saturation, which is tailored to $H^0_\bullet$ and computationally more efficient as it discards all higher cohomologies.
This paper presents a categorical setup which yields unified elementary proofs of both theorems.

A central notion in this paper is that of the linear regularity.
We use it in Corollary~\ref{coro:saturating_power} for the convergence analysis of the inductive limit defining the ideal transform and in Corollary~\ref{coro:defect_of_saturation} to give the number of recursion steps in our new algorithm.
Like the \nameft{Tate} resolution, the \nameft{Castelnuovo-Mumford} regularity incorporates all higher cohomologies.
And again, the linear regularity is an adaption to $H^0_\bullet$ which discards all higher cohomologies.
It follows that the linear regularity is smaller or at most equal to the \nameft{Castelnuovo-Mumford} regularity.
This provides another reason why computing the purely linear saturation is more efficient than computing the Tate resolution.

One application is the computation of global sections.
More precisely, let $\F$ be a coherent sheaf on $\PP_B^n$ and $\pi:\PP^n_B \to \Spec B$ the natural projection.\footnote{%
The base $\Spec B$ might even serve as the ambient space of a geometric quotient, e.g., if $B$ is the \nameft{Cox} ring of a toric variety.}
The direct image sheaf $H^0(\F) := \pi_* \F$ over $\Spec B$ is the degree zero part (cf.\ Algorithm~\ref{algo:degree_i}) of $H^0_{\geq 0}(\F)$.
For example, if $\O_X \subset \O_{\PP_B^n}$ denotes the structure sheaf  of a subscheme $X \subset \PP_B^n$, then computing $\pi_* \O_X$ is the geometric form of eliminating all $n+1$ homogeneous coordinates
$x_0,\ldots,x_n$ from the defining equations of $X \subset \PP_B^n$.

A computer model of the \nameft{Abel}ian category $\Coh \PP_B$ of coherent sheaves on $\PP_B^n$ must incorporate the objects \emph{and} the morphisms\footnote{The more involved modeling of the morphisms is of no relevance for this paper (cf.~\cite{BL_GabrielMorphisms}).}.
We represent an object $\F \in \Coh \PP_B^n$ by a finitely presented graded $S$-modules $M$, such that $\F:=\widetilde{M}$ is the sheafification of $M$.

The ideal transform is defined to be the inductive limit $D_\m(M)$ of the graded modules $\GradedHom(\m^\ell,M)$, where $\m = \langle x_0,\ldots,x_n \rangle \lhd S$ is the irrelevant ideal.
The equivalence $H^0_\bullet\left(\widetilde{\,\cdot\,}\right) \simeq \varinjlim_\ell \GradedHom\left(\m^\ell,\,\cdot\,\right)$, reproved elementarily as Theorem \ref{thm:Dm}, implies that $H^q_\bullet\left(\widetilde{\,\cdot\,}\right) \simeq \varinjlim_\ell \GradedExt^q\left(\m^\ell,\,\cdot\,\right)$ for the higher derived cohomology functors \cite[20.4.4]{BrSh}.

Another description of the cohomology functors $H^q_\bullet$ arose from the BGG-correspondence \cite{BGG}.
It is a triangle equivalence between the bounded derived category of $\Coh \PP^n_B$ (originally over a base field $B$) and the stable category of finitely generated graded $E$-modules over the exterior algebra $E$, the \nameft{Koszul} dual $B$-algebra of $S$.
Since $E$ is a \nameft{Frobenius} algebra, this stable category is easily seen to be triangle equivalent to the homotopy category of so-called \nameft{Tate} complexes.
A constructive description of the composition of these two triangle equivalences was given in \cite{EFS,DE}.
The treatment of the relative BGG-correspondence in \cite{ES} does not only describe the coherent sheaf cohomologies $H^q(\widetilde{M}) = R^q \pi_* \widetilde{M}$ as $B$-modules, but also provides a concrete realization of the direct image complex $R \pi_* \widetilde{M}$.
However, in this approach even the computation of global sections $H^0(\widetilde{M})$ in the relative case relies a priori on the entire \nameft{Tate} resolution.

The bottom complex $E_1^{\geq d,0} \! \left(\T^{\geq d}(M)\right)$ on the first page of the spectral sequence of the \nameft{Tate} resolution $\mathbf{T}^{\geq d}(M)$ is a linear complex which corresponds to $H^0_{\geq d}(\widetilde{M})$.
We define a new so-called purely linear saturation functor $\bSd$, which is computationally far more economic than the \nameft{Tate} functor $\mathbf{T}^{\geq d}$.
In Theorem~\ref{thm:tate_monad}, Proposition~\ref{prop:monads_computing_H0}, and Corollary\ref{coro:monads_computing_H0} we prove that $\bSd$ computes $H^0_{\geq d}(\widetilde{M})$, and hence $E_1^{\geq d,0} \! \left(\mathbf{T}^{\geq d}(M)\right)$.
The point is that we can compute $\bSd$ \emph{without} computing $\mathbf{T}^{\geq d}$.
This statement is not obvious in the relative case (cf.~Remark~\ref{rmrk:E1}).
Furthermore, the linear regularity of $M$ gives the precise number of recursion steps needed to achieve saturation.
Since computing $\bSd$ relies on \nameft{Gröbner} bases over the exterior algebra $E$ of \emph{finite} rank over $B$ the involved algorithms are, for many examples, faster than the ones for the ideal transform.
The latter involve \nameft{Gröbner} bases over the polynomial ring $S$ of \emph{infinite} rank over $B$.

In order to develop a unified proof that both functors $D_{\m, \geq d}$ and $\bSd$ compute $H^0_{\geq d}(\widetilde{\,\cdot\,})$ we need an appropriate categorical setup.
Abstractly, the category $\Coh\PP_B^n$ of coherent sheaves on $\PP_B^n$ is equivalent to the \nameft{Serre} quotient category $\A/\C$ of the \nameft{Abel}ian category $\A$ of finitely presented graded $S$-modules modulo a certain subcategory $\C$.
The necessary categorical language is summarized in Section~\ref{sec:cat}.
In Section~\ref{sec:sheaves} we show that the categories $\A$ and $\C$ can be replaced by their respective full subcategories of modules which vanish in degrees $<d$ for an arbitrary but fixed $d \in \Z$ (cf.\ Proposition~\ref{prop:truncated_quotients_are_sheaves}).
The $\A$-endofunctor $M \mapsto H^0_{\geq d}(\widetilde{M})$ is a special case of what we call a \nameft{Gabriel} monad, which we characterized in \cite{BL_Monads} by a short set of properties.
By verifying this short list of properties for the two functors $D_{\m, \geq d}$ and $\bSd$ we prove that they compute $H^0_{\geq d}(\widetilde{\,\cdot\,})$.

Two further applications rely on constructivity of the \nameft{Gabriel} monad, and now become algorithmically accessible for the category $\Coh \PP_B^n$:
First, the \nameft{Serre} quotient category $\A/\C$ becomes constructively \nameft{Abel}ian once $\A$ is constructively \nameft{Abel}ian \cite[Appendix~D]{BL_GabrielMorphisms}.\footnote{However, the approach using \nameft{Gabriel} morphisms in \cite{BL_GabrielMorphisms} seems computationally faster.}
Second, the computability of the bivariate $\Hom$ and $\Ext^i$ functors in $\A/\C$ now reduces to the computability of $\Hom$ and $\Ext^i$ in $\A$ (modulo a directed colimit process if $i>0$) \cite{BL_ExtComputability}.

\section{Preliminaries on \texorpdfstring{\nameft{Serre}}{Serre} Quotient Categories} \label{sec:cat}

A non-empty full subcategory $\mathcal{C}$ of an \nameft{Abel}ian category $\A$ is called \textbf{thick} if it is closed under passing to subobjects, factor objects, and extensions.
In this case the \textbf{\nameft{Serre} quotient category} $\A/\C$ is a category with the same objects as $\A$ and $\Hom$-groups defined by the directed colimit
\[
  \Hom_{\A/\C}(M,N) := \varinjlim_{\substack{M' \le M, N' \le N,\\ M/M', N' \in \C}} \Hom_\A(M',N/N')\mbox{.}
\]
The \textbf{canonical functor} $\QQ:\A \to \A/\C$ is defined to be the identity on objects and maps a morphism $\phi \in \Hom_\A(M,N)$ to its class in the directed colimit $\Hom_{\A/\C}(M,N)$.
The category $\A/\C$ is \nameft{Abel}ian and the canonical functor $\QQ: \A \to \A/\C$ is exact.
An object $M \in \A$ is called \textbf{$\C$-saturated} if $\Ext_\A^0(C,M)\cong\Ext_\A^1(C,M)\cong0$ for all $C\in\C$, i.e., $M$ has no nonzero subobjects in $\C$ and every extension of an object $C \in \C$ by $M$ is trivial.
Denote by $\Sat_\C(\A)\subset\A$ the full subcategory of $\C$-saturated objects and by $\iota: \Sat_\C(\A) \hookrightarrow \A$ its full embedding.
A complex $F$ in $\Sat_\C(\A)$ is exact if and only if $\iota(F)$ has homology in $\C$.

A thick subcategory $\C\subset\A$ is called \textbf{localizing} if the canonical functor $\QQ: \A \to \A/\C$ admits a right adjoint $\SS:\A/\C \to \A$, called the \textbf{section functor} of $\QQ$.
In this case, the image of $\SS$ is contained in $\Sat_\C(\A)$ and $\A/\C\xrightarrow{\SS}\SS(\A/\C)\hookrightarrow\Sat_\C(\A)$ are equivalences of categories.
The $\Hom$-adjunction $\Hom_{\A/\C}(\QQ(M),\QQ(N))\cong\Hom_{\A}(M,(\SS\circ\QQ)(N))$ allows to compute $\Hom$-groups in $\A/\C$ if they are computable in $\A$ and the monad $\SS\circ\QQ$ is computable.
In particular, this avoids computing the directed colimit in the definition of $\Hom_{\A/\C}$.
We call any monad equivalent to $\SS\circ\QQ$ a \textbf{\nameft{Gabriel} monad} (of $\A$ w.r.t.\ $\C$).
The following theorem characterizes \nameft{Gabriel} monads.

\begin{theorem}[{\cite[Thm.~3.6]{BL_Monads}}\footnote{Thm.~4.6 in \texttt{arXiv} version.}]\label{thm:Csaturating}
  Let $\C \subset \A$ be a localizing subcategory of the \nameft{Abel}ian category $\A$ and $\iota: \Sat_\C(\A) \hookrightarrow \A$ the full embedding.
  An endofunctor $\WW:\A \to \A$ together with a natural transformation $\eta:\Id_\A \to \WW$ is a \nameft{Gabriel} monad (of $\A$ w.r.t.\ $\C$) if and only if the following five conditions hold:
  \begin{enumerate}
    \item $\C \subset \ker \WW$, \label{thm:Csaturating:ker}
    \item $\WW(\A) \subset \Sat_\C(\A)$, \label{thm:Csaturating:im}
    \item the corestriction $\cores_{\Sat_\C(\A)} \WW$ of $\WW$ to $\Sat_\C(\A)$ is exact, \label{thm:Csaturating:exact}
    \item $\eta \WW = \WW \eta: \WW \to \WW^2$, and \label{thm:Csaturating:idem}
    \item $\eta\iota: \iota \to \WW \iota$ is a natural isomorphism. \label{thm:Csaturating:eta}
  \end{enumerate}
\end{theorem}

In Sections~\ref{sect:ideal_transforms} and \ref{sec:saturate_linear_complex} we utilize this theorem to prove that certain functors are \nameft{Gabriel} monads of the category of coherent sheaves on the relative projective space $\PP_B^n$, and thus compute the (truncated) module of twisted global sections.
However, this theorem, abstract as it is, can be applied to categories of coherent sheaves of more general schemes.

\section{Graded modules over the free polynomial ring}\label{sec:modules}

For the rest of the paper let $B$ denote a \nameft{Noether}ian commutative ring with $1$, $V$ a free $B$-module of rank $n+1$, $W := V^* = \Hom_B(V,B)$ its $B$-dual, and $x_0,\ldots,x_n$ a free generating set of the $B$-module $W$.
Set
\[
  S := \Sym_B(W)=B[V]=B[x_0,\ldots,x_n]
\]
to be the free polynomial ring over $B$ in the $n+1$ indeterminates $x_0,\ldots,x_n$.
Setting $\deg(x_j)=1$ turns $S$ into a positively graded ring $S=\bigoplus_{i \geq 0} S_i$ where $S_i$ is the set of homogeneous polynomials of degree $i$ in $S$.
Define the irrelevant ideal
\[
  \m:=S_{>0}=\langle x_0,\ldots,x_n \rangle \lhd S \mbox{.}
\]
The isomorphism $B = S_0 \cong S/\m$ endows $B$ with a natural graded $S$-module structure.

To make the statements of this paper constructive, the ring $S$ needs to have a Gröbner bases algorithm.
This is the case if $B$ has effective coset representatives \cite[\S 4.3]{AL}, i.e., for every ideal $I\subset B$ we can determine a set $T$ of coset representatives of $B/I$, such that for every $b \in B$ we can compute a unique $t \in T$ with $b+I=t+I$.

We denote by $\Sfpmod$ the category of (non-graded) finitely presented  $S$-modules and by $\Sfpgrmod$ the category of finitely presented \emph{graded} $S$-modules.
Further we denote by
\[
  \Sfpgrmodd \subset \Sfpgrmod
\]
the full subcategory of all modules $M$ with $M = M_{\geq d}$.
Define the shift autoequivalence on $\Sfpgrmod$ by $M(i)_j := M_{i+j}$ for all $i\in\Z$; it induces an endofunctor on the subcategory $\Sfpgrmodd$ for $i\le0$.

\begin{algo} \label{algo:degree_i}
  We briefly describe how to compute the $i$-th homogeneous part of an $M \in \Sfpgrmod$:
  Such a module is realized on the computer as the cokernel of a graded free $S$-presentation
  \[
    M \stackrel{\pi}{\twoheadleftarrow} \bigoplus_k S(g_k) \stackrel{\mathtt{m}}\leftarrow \bigoplus_\ell S(r_\ell) \mbox{.}
  \]
  The image of the graded submodule $\bigoplus_{k, i + g_k \geq 0} S_{\geq i + g_k}(g_k) \leq \bigoplus_k S(g_k)$ under $\pi$ is the graded submodule $\langle M_i \rangle \leq M$, which we compute as the kernel of the cokernel of the restricted map $M \leftarrow \bigoplus_{k, i+g_k \geq 0} S_{\geq i + g_k}(g_k)$.
  Computing a free $S$-presentation $\langle M_i \rangle \twoheadleftarrow \bigoplus_k S(i) \stackrel{\mathtt{m}_i}{\leftarrow}  \bigoplus_\ell S(r_\ell')$ of $\langle M_i \rangle$ thus involves two successive syzygy computations as explained in \cite[(10) in the proof of Theorem~3.4]{BL}.
  To get a free $B$-presentation of $M_i$ we just need to tensor the last exact sequence with $B=S/\m$ over $S$, which corresponds to extracting the degree $0$ relations in the reduced \nameft{Gröbner} basis of the $S$-matrix of relations $\mathtt{m}_i$.
\end{algo}

\subsection{Internal and external \texorpdfstring{$\Hom$}{Hom} functors}

Let $M,N \in \Sfpgrmod$.
Then the $\Hom$-group $\Hom_\Sfpmod(M,N)$ of their underlying modules in $\Sfpmod$ is again naturally graded.
This induces internal $\Hom$ functors
\[
  \GradedHom:\Sfpgrmod^\mathrm{op}\times\Sfpgrmod\to\Sfpgrmod
\]
in the category $\Sfpgrmod$ and 
\[
  \Hom_{\geq d}:\Sfpgrmod_{\ge d}^\mathrm{op}\times\Sfpgrmodd\to\Sfpgrmodd
\]
in $\Sfpgrmodd$.
These internal $\Hom$ functors are algorithmically computable if $B$ has effective coset representatives (cf., e.g., \cite[\S 4.3]{AL} and \cite[§3.3]{BL}).

The (external) $\Hom$-groups of the category $\Sfpgrmod$ are finitely generated $B$-modules.
They can be recovered as the graded part of degree $0$ of the corresponding internal $\Hom$'s:
\begin{align*}
  \Hom(M,N) := \Hom_{\Sfpgrmod}(M,N) &\cong \GradedHom(M,N)_0 \mbox{,} \\
  \Hom_{\Sfpgrmodd}(M,N) &\cong \Hom_{\geq d}(M,N)_0 \quad \mbox{for $d \leq 0$.}
\end{align*}
In particular, $\Hom_{\Sfpgrmod}(S, M) \cong M_0$ and $\Hom_{\Sfpgrmodd}(S, M) \cong M_0$ for $d \leq 0$.

Dealing with $d>0$ would enforce further case distinctions.
For example, $B \cong S/\m$ lies in $\Sfpgrmodd$ only if $d \leq 0$.

\medskip
\framebox{
\begin{minipage}[c]{0.9\textwidth}
\begin{center}
  Till the end of Section~\ref{sect:ideal_transforms} we assume that $d \leq 0$.
\end{center}
\end{minipage}
}

\begin{rmrk} \label{rmrk:etal}
  Applying $\GradedHom(-,M)$ to the short exact sequence $S/\m^\ell \twoheadleftarrow S \hookleftarrow \m^\ell$ yields
  \begin{equation} \label{etal} \tag{$\eta^\ell_M$}
      \GradedHom(S/\m^\ell,M) \hookrightarrow M \xrightarrow{\eta^\ell_M} \GradedHom(\m^\ell,M) \twoheadrightarrow \Ext_\bullet^1(S/\m^\ell,M)
  \end{equation}
  as part of the long exact contravariant $\Ext_\bullet$-sequence.
  We will repeatedly refer to this exact sequence as well as to the $\ell=1$ case
  \begin{equation} \label{eta1} \tag{$\eta^1_M$}
      \GradedHom(B,M) \hookrightarrow M \xrightarrow{\eta^1_M} \GradedHom(\m,M) \twoheadrightarrow \Ext_\bullet^1(B,M) \mbox{.}
  \end{equation}
\end{rmrk}

\subsection{Quasi-zero modules}

Let $\Sfpgrmod^0$ denote the thick subcategory of \textbf{quasi-zero} modules, i.e., those with $M_{\geq \ell}=0$ for $\ell$ large enough.
Analogously, we denote by $\Sfpgrmod^0_{\geq d}$ the \emph{localizing} (cf.~Theorem~\ref{thm:Dm}) subcategory of quasi-zero modules in $\Sfpgrmodd$.

\begin{rmrk} \label{rmrk:TorExtQZ}
  For $M \in \Sfpgrmod$. Then for all $\ell \geq 0$
  \begin{enumerate}
    \item $\Tor^S_i(S/\m^\ell,M)_\bullet \in \Sfpgrmod^0\,$ for all $i \geq 0$. \label{rmrk:TorExtQZ.Tor}
    \item $\Ext_\bullet^j(S/\m^\ell,M) \in \Sfpgrmod^0\,$ for all $j \geq 0$. \label{rmrk:TorExtQZ.ExtofSmodm}
    \item $\Ext_\bullet^j(\m^\ell,M) \in \Sfpgrmod^0\,$ for all $j \geq 1$. \label{rmrk:TorExtQZ.Extofm}
  \end{enumerate}
\end{rmrk}
\begin{proof}
  The existence of a finitely generated free resolution of the first argument $S/\m^\ell$ (and hence of $\m^\ell$) implies that all the above derived modules lie in $\Sfpgrmod$.
  By applying $S/\m^\ell \otimes_S -$ to a projective resolution of $M$ and $\GradedHom(S/\m^\ell, -)$ to an injective resolution of $M$ shows that the ideal $\m^\ell \lhd S$ annihilates $\Tor^S_i(S/\m^\ell,M)_\bullet$ and $\Ext_\bullet^j(S/\m^\ell,M)$, which implies that they are also finitely generated $S/\m^\ell$-modules, proving \eqref{rmrk:TorExtQZ.Tor} and \eqref{rmrk:TorExtQZ.ExtofSmodm}.
  The existence of the connecting isomorphisms $\Ext_\bullet^j(\m^\ell,M) \cong \Ext_\bullet^{j+1}(S/\m^\ell,M)$ ($j \geq 1$) finally implies \eqref{rmrk:TorExtQZ.Extofm}.
\end{proof}

\begin{rmrk} \label{rmrk:Kleene}
  The use of the nonconstructive injective resolution in the previous proof is an example of an admissible use of nonconstructive arguments in an otherwise constructive setup to prove statements which neither involve existential quantifiers nor disjunctions (so-called negative formulae):
  $\Ext_\bullet(S/\m^\ell,M)$ has two descriptions.
  The nonconstructive one in the proof and the constructive one in which $\Hom(-,M)$ is applied to a finite free resolution of $S/\m^\ell$.
  Although the isomorphism between the two descriptions is not constructive it is ``good enough'' for transferring the property we want to establish.
\end{rmrk}

\subsection{Regularity, linear regularity, and relation to \texorpdfstring{$\Tor$}{Tor} and \texorpdfstring{$\Ext$}{Ext}}\label{subsec:tor}

For convenience of the reader we recall the definition of the \textbf{\nameft{Castelnuovo-Mumford} regularity} in the relative case from \cite[§2]{ES}.
For any quasi-zero graded $S$-module $N$ define
\[
  \reg N := \max \{d \in \Z \mid N_d \neq 0\} \mbox{.}
\]
The regularity of the zero module is set to $-\infty$.
Then, for $M \in \Sfpgrmod$ the $S$-module $\Tor^S_i(B,M)_\bullet$ is quasi-zero and
\[
  \reg M := \max \{ \reg \Tor^S_i(B,M)_\bullet - i \mid i=0,\ldots,n+1 \} \mbox{.}
\]

Equivalently, one can define
\[
  \reg M := \max \{ \reg H_\m^j(M) + j \mid j=0,\ldots,n+1 \}
\]
using the local cohomology modules $H^j_\m(M) = \varinjlim_\ell \Ext_\bullet^j(S / \m^\ell, M)$ (cf., e.g., \cite[Prop.~2.1]{ES}).\footnote{This definition clarifies the relation to two other regularity notions: The \textbf{geometric regularity} is defined by $\operatorname{g-reg} M := \max \{ \reg H_\m^j(M) + j \mid j=\mathbf{1},\ldots,n+1 \}$ and the \textbf{regularity of the sheafification} $\reg \widetilde{M} := \max \{ \reg H_\m^j(M) + j \mid j=\mathbf{2},\ldots,n+1 \}$.}
In fact only $\ell=1$ in this sequential colimit is relevant for us.
To see this we need the following result, which we also use in the proof of our key Lemma~\ref{lma:HRR_Ext}.
\begin{lemm} \label{lemm:Betti=Bass}
  There exists a natural isomorphism
  \[
    \Tor^S_i(B,M)_\bullet \cong \Ext_\bullet^{n+1-i}(\wedge^{n+1}V, M) \mbox{.}
  \]
\end{lemm}
\begin{proof}
  The $\Tor$-$\Ext$ spectral sequence $\Tor^S_{-p}(\Ext_\bullet^q(\wedge^{n+1}V, S), M)_\bullet \Rightarrow \Ext_\bullet^{p+q}(\wedge^{n+1} V, M)$ collapses since $\Ext_\bullet^q(\wedge^{n+1}V,S)=0$ for $q\not=n+1$ and $\Ext_\bullet^{n+1}(\wedge^{n+1}V,S)=B$.
\end{proof}
When $B=k$ is a field this Lemma becomes the intrinsic and rather generalizable form of the equality between the graded \nameft{Betti} numbers $\beta_{ij} := \dim_k \Tor^S_i(B,M)_j$ and the graded \textbf{\nameft{Bass} numbers}:
\[
  \mu_{n+1-i,j-n-1} := \dim_k \Ext_\bullet^{n+1-i}(\wedge^{n+1}V,M)_j \mbox{.}
\]

\begin{rmrk}
Lemma~\ref{lemm:Betti=Bass} and the noncanonical isomorphism $\wedge^{n+1}V \cong B(n+1)$ yield
\[
  \reg M = \max \{ \reg \Ext_\bullet^j(B, M) +j \mid j=0,\ldots,n+1 \} \mbox{.}
\]
\end{rmrk}

The value of the following definition will start to become obvious in Proposition~\ref{prop:saturation} in the next subsection.
\begin{defn}\label{defn:lin_reg}
  Define the \textbf{linear regularity} of $M \in \Sfpgrmod$ to be
  \[
    \linreg M = \max \{\reg \Ext_\bullet^j(B, M) \mid j=0,1 \} \in \Z \cup \{-\infty\} \mbox{.}
  \]
  Analogously, the \textbf{$d$-th truncated linear regularity} of $M \in \Sfpgrmod_{\geq d}$ is defined by
  \[
    \linregd M = \max \{\reg \Ext_{\geq d}^j(B, M) \mid j=0,1 \} \in \Z_{\geq d} \cup \{-\infty\} \mbox{,}
  \]
  for $d \leq 0$ where $\Ext_{\geq d}^j := \Ext_{\Sfpgrmodd}^j \simeq (\Ext_\bullet^j)_{\geq d}$.
\end{defn}

Note that $\linreg = \reg$ on $\Sfpgrmod^0$ and $\linreg \leq \reg$ on $\Sfpgrmod$.
\begin{exmp}
    $\linreg S/\m^{\ell+1} = \reg S/\m^{\ell+1} = \ell = \linreg \m^{\ell + 1} < \reg \m^{\ell + 1} = \ell + 1$.
\end{exmp}

The motivation behind introducing $\linreg$ is that it offers an upper bound in the saturation algorithms discussed below, where the use of the (often enough much larger) regularity would be a waste of computational resources.

\subsection{Saturated modules}

The equivalent conditions \eqref{prop:saturation.A} and \eqref{prop:saturation.S-m} in the following proposition are computationally effective characterizations of saturated modules.

\begin{prop} \label{prop:saturation}
  For $M \in \Sfpgrmod$ the following are equivalent:
  \begin{enumerate}
    \item $M$ is saturated w.r.t.\  $\Sfpgrmod^0$; \label{prop:saturation.sat}
    \item $\Ext_\bullet^0(S / \m^\ell,M) = 0$ and $\Ext_\bullet^1(S / \m^\ell,M) = 0$ for all $\ell \geq 0$; \label{prop:saturation.Sml}
    \item The natural map $\eta^\ell_M := \GradedHom(S \hookleftarrow \m^\ell, M): M \to \GradedHom(\m^\ell,M)$ is an isomorphism for all $\ell \geq 0$; \label{prop:saturation.S-ml}
    \item $\Ext_\bullet^0(B,M) = 0$ and $\Ext_\bullet^1(B,M) = 0$;\footnote{%
      Conditions \eqref{prop:saturation.A} and \eqref{prop:saturation.S-m} are in their use of \nameft{Gröbner} bases algorithmically equivalent.
      Computing them only involves the first two morphisms in the \nameft{Koszul} resolution of $B$ (and then tensoring their duals with $M$).
      One might be tempted to expect that \eqref{prop:saturation.A} is always algorithmically superior to condition \eqref{prop:saturation.Tor}, which seem to involve an $n+1$-term resolution of either $B$ or of $M$.
      However, one can easily construct examples of $M \in \Sfpgrmod$, where condition \eqref{prop:saturation.Tor} is algorithmically superior, e.g., if the resolution of $M$ terminates after few steps, long before reaching step $n$.
      } \label{prop:saturation.A}
    \item The natural map $\eta^1_M := \GradedHom(S \hookleftarrow \m, M): M \to \GradedHom(\m,M)$ is an isomorphism; \label{prop:saturation.S-m}
    \item $\Tor^S_{n+1}(B,M)_\bullet = 0$ and $\Tor^S_n(B,M)_\bullet = 0$; \label{prop:saturation.Tor}
    \item $\linreg M = -\infty$. \label{prop:saturation.linreg}
  \suspend{enumerate}
  And if the base ring $B$ is a field the above is also equivalent to:
  \resume{enumerate}
    \item The projective dimension $\operatorname{pd} M \leq n-1$. \label{prop:saturation.pd}
  \end{enumerate}
\end{prop}

In the proof of this proposition, we use the following simple remark.
\begin{rmrk}\label{rmrk:m.tensor}
  The kernel $K$ of the epimorphism $\m^\ell \twoheadleftarrow \otimes^\ell \m$ is concentrated in degree $\ell$.
  To see this note that any homogeneous element in $\otimes^\ell \m$ of degree $m > \ell$ which is the tensor product of monomials can be brought to the normal form $x_{i_1} \otimes_B \cdots \otimes_B x_{i_{\ell-1}} \otimes_B x^\mu$ with $i_1 \leq \cdots \leq i_{m-1} \leq \min\{ i \mid \mu_i \neq 0 \}$ and $|\mu| = m-\ell+1$.
  This kernel $K$ is free over $B$ of rank $(n+1)^\ell - \binom{n+\ell}{n}$ as the kernel of the $B$-epimorphism $\Sym^\ell W \twoheadleftarrow \otimes^\ell W$.
\end{rmrk}

\begin{proof}[Proof of Proposition~\ref{prop:saturation}]
  \item[\eqref{prop:saturation.Sml} $\Leftrightarrow$ \eqref{prop:saturation.S-ml}:]
    The claim is obvious from the \eqref{etal}-sequence in Remark~\ref{rmrk:etal}.
  \item[\eqref{prop:saturation.A} $\Leftrightarrow$ \eqref{prop:saturation.S-m}:]
    This is a special case of the equivalence \eqref{prop:saturation.Sml} $\Leftrightarrow$ \eqref{prop:saturation.S-ml} for $\ell = 1$.
  \item[\eqref{prop:saturation.A} $\Leftrightarrow$ \eqref{prop:saturation.Tor}:]
    This is the statement of Lemma~\ref{lemm:Betti=Bass} for $i=n+1$ and $i=n$.
  \item[\eqref{prop:saturation.A} $\Leftrightarrow$ \eqref{prop:saturation.linreg}:]
    By definition of $\linreg$.
  \item[\eqref{prop:saturation.sat} $\Rightarrow$ \eqref{prop:saturation.A}:]
    This follows directly from the definition of saturated objects (cf.~Section~\ref{sec:cat}), as $B\in\C=\Sfpgrmod^0$.
  \item[\eqref{prop:saturation.S-m} $\Rightarrow$ \eqref{prop:saturation.S-ml}:]
    Applying the $\ell$-th power of $\GradedHom(S \hookleftarrow \m,-)$ to $M$ and taking the diagonal in the $\ell$-dimensional cube yields the isomorphism
    \[
      \phi := M \xrightarrow{\sim} \GradedHom(\otimes^\ell \m, M)
    \]
    by the adjunction between $\otimes$ and $\GradedHom$.
    This isomorphism can be written as the composition
    \[
      \GradedHom(S \hookleftarrow \m^\ell \twoheadleftarrow \otimes^\ell \m,M)
      =
      \left( M \xrightarrow{\psi} \GradedHom(\m^\ell,M) \xrightarrow{\chi} \GradedHom(\otimes^\ell \m,M) \right) \mbox{.}
    \]
    The homomorphism $\chi$ is a monomorphism since $\GradedHom$ is left exact and an epimorphism since $\chi\circ\psi=\phi$ is an isomorphism.
    Hence, $\chi$ is isomorphism and thus $\psi$ is an isomorphism.
  \item[\eqref{prop:saturation.Sml} $\Rightarrow$ \eqref{prop:saturation.sat}:]
    It is clear that any $N \in \Sfpgrmod^0$ is an epimorphic image of $\bigoplus_{i\in I} (S/\m^{a_i})(b_i)$ for a finite set $I$ and suitable $a_i$ and $b_i$.
    Denote the kernel of $N \twoheadleftarrow \bigoplus_{i} (S/\m^{a_i})(b_i)$ by $K$.
    Applying $\GradedHom(-,M)$ to $N \twoheadleftarrow \bigoplus_{i} (S/\m^{a_i})(b_i) \hookleftarrow K$ yields as parts of the long exact sequence
    \[
      \GradedHom(N,M)
      \hookrightarrow
      \underbrace{\GradedHom\left(\bigoplus_{i} (S/\m^{a_i})(b_i),M\right)}_{\cong 0} \mbox{,}
    \]
    and
    \[
      \GradedHom(K,M)
      \to
      \Ext_\bullet^1(N,M)
      \to
      \underbrace{\Ext_\bullet^1 \left(\bigoplus_{i} (S/\m^{a_i})(b_i),M\right)}_{\cong 0} \mbox{.}
    \]
    The first part implies that $\GradedHom(-,M) = 0$ on $\Sfpgrmod^0$.
    In particular $\GradedHom(K,M)=0$ since $K \in \Sfpgrmod^0$.
    Combining this and the second part implies that $\Ext_\bullet^1(-,M)$ vanishes on $\Sfpgrmod^0$.
  \item [\eqref{prop:saturation.Tor} $\Leftrightarrow$ \eqref{prop:saturation.pd}:]
    If $B$ is a field then there exists a finite free (and not merely relatively free) presentation $M \twoheadleftarrow F_\bullet$ with $\Tor^S_i(B,M)_\bullet$ isomorphic to the head of $F_i$.
\end{proof}

\begin{coro}\label{coro:saturation_0}
  For $M \in \Sfpgrmodd$ the following are equivalent (recall, $d \leq 0$):
  \begin{enumerate}
    \item $M$ is saturated w.r.t.\  $\Sfpgrmodd^0$; \label{coro:saturation_0.sat}
    \item $\Ext_{\geq d}^0(S/\m^\ell,M) = 0$ and $\Ext_{\geq d}^1(S/\m^\ell,M) = 0$ for all $\ell \geq 0$; \label{coro:saturation_0.Sml}
    \item The natural map $\eta^\ell_M := \Hom_{\geq d}\left( S \hookleftarrow \m^\ell, M \right) : M \to \Hom_{\geq d}(\m^\ell,M)$ is an isomorphism for all $\ell \geq 0$; \label{coro:saturation_0.S-ml}
    \item $\Ext_{\geq d}^0(B,M) = 0$ and $\Ext_{\geq d}^1(B,M) = 0$; \label{coro:saturation_0.A}
    \item The natural map $\eta^1_M := \Hom_{\geq d}\left( S \hookleftarrow \m, M \right) : M \to \Hom_{\geq d}(\m,M)$ is an isomorphism; \label{coro:saturation_0.S-m}
    \item $\Tor^S_{n+1}(B(n+1),M)_{\geq d} = 0$ and $\Tor^S_n(B(n+1),M)_{\geq d} = 0$; \label{prop:saturation_0.Tor}
    \item $\linregd M = -\infty$. \label{coro:saturation_0.linreg}
  \end{enumerate}
\end{coro}

\section{Ideal transforms}\label{sect:ideal_transforms}

Recall, the \textbf{$\m$-transform} of $M \in \Sfpgrmod$ is the (not necessarily finitely generated) graded $S$-module defined by the sequential colimit
\[
  D_\m := \varinjlim_\ell\GradedHom(\m^\ell,-):\Sfpgrmod\to\SfpgrMod.
\]

On $\Sfpgrmodd$ the \textbf{$d$-truncated $\m$-transform}  (recall, $d \leq 0$)
\[
  D_{\m,\geq d} := \varinjlim_\ell \Hom_{\geq d}(\m^\ell,-):\Sfpgrmodd\to\Sfpgrmodd
\]
is an endofunctor.
This is a simple corollary of the Lemma~\ref{lemm:Dm_colimit} below.

\begin{defn} \label{defn:saturation_interval}
  We define the \textbf{saturation interval} of $M \in \Sfpgrmodd$ to be
  \[
    I_{\geq d}(M) := [\delta^0_{M,d}, \delta^1_{M,d}] \cap \Z \subset \Z_{\geq 0} \mbox{,}
  \]
  where $\delta^0_{M,d} := \max \{ \reg \Hom_{\geq d}(B,M) - d + 1, 0 \}$ and $\delta^1_{M,d} := \max \{ \linregd - d + 1, 0 \}$.
\end{defn}

The saturation interval plays a role in the following convergence analysis and the definition of its upper bound is a further motivation for the linear regularity.
\begin{lemm} \label{lemm:Dm_colimit}
  For each $M \in \Sfpgrmodd$ the sequential colimit defining the $\m$-transform is finite.
  More precisely, there exists a nonnegative integer $\delta_{M,d} \in I_{\geq d}(M)$ such that the induced maps $\Hom_{\geq d}(\m^\ell, M) \to \Hom_{\geq d}(\m^{\ell+1},M)$ are isomorphisms for all $\ell \geq t$ iff $t \geq \delta_{M,d}$.
  In particular, the natural map
  \[
    \Hom_{\geq d}(\m^\ell,M) \to D_{\m,\geq d}(M)
  \]
  is an isomorphism iff $\ell \geq \delta_{M,d}$.
\end{lemm}
\begin{proof}
  The short exact sequence $B(-\ell)^{\oplus ?} \cong \m^\ell/\m^{\ell+1} \twoheadleftarrow \m^\ell \hookleftarrow \m^{\ell+1}$ induces for $M \in \Sfpgrmod$ the exact contravariant $\Ext_\bullet$-sequence of which the first four terms are
  \[
    \GradedHom(B,M)^{\oplus ?} (\ell) \hookrightarrow \GradedHom(\m^\ell,M) \to \GradedHom(\m^{\ell+1},M) \to \Ext_\bullet^1(B,M)^{\oplus ?} (\ell) \mbox{.}
  \]
  By Remark~\ref{rmrk:TorExtQZ}.\eqref{rmrk:TorExtQZ.ExtofSmodm} both $\GradedHom(B,M)$ and $\Ext_\bullet(B,M)$ are quasi-zero.
  Hence, there exists a $\delta_{M,d} \in I_{\geq d}(M)$ such that the truncated morphisms $\Hom_{\geq d}(\m^\ell, M) \to \Hom_{\geq d}(\m^{\ell+1},M)$ become isomorphisms in $\Sfpgrmodd$ for $\ell \geq t$ iff $t \geq \delta_{M,d}$.
\end{proof}

In particular, once $B$ has effective coset representatives, $D_{\m,\geq d}$ is algorithmically computable on objects and morphisms, since the internal $\Hom$ functor $\Hom_{\ge d}$ is.

\begin{defn} \label{defn:defect_of_saturation}
  We call $\delta_{M,d} \in I_{\geq d}(M)$ from Lemma~\ref{lemm:Dm_colimit} the \textbf{defect of saturation} of $M$.
\end{defn}

\begin{exmp}
  Note that $1=\delta_{\m(-t),0} \in I_{\geq 0}(\m(-t)) = [0,\linreg_{\geq 0} \m(-t)+1] = [0,t+1]$ for all $t \in \Z_{\geq 0}$.
  In other words, the maximum of $I_{\geq d}(M)$ can be an arbitrarily bad upper bound for $\delta_{M,d}$.
\end{exmp}

\begin{exmp}
  For $M=S \oplus B(-t)$ and $t \geq 0$ we compute $\GradedHom(B,M)=B(-t)$ and $\Ext_\bullet^1(B,M)=B(-t+1)^{n+1}$ (for $n > 0$).
  Hence $\delta_{M,0}^0=t+1=\delta_{M,0}^1 = \delta_{M,0}$ is the defect of saturation.
  Thus, for certain examples factoring out the $\Sfpgrmodd^0$-torsion part of $M$ a priori could be beneficial.
\end{exmp}

The natural transformation
\[
  \eta_M := \varinjlim_\ell \left( \eta^\ell_M: M \to \Hom_{\geq d}(\m^\ell,M) \right): M \to D_{\m,\geq d}(M)
\]
is induced by applying $\Hom_{\geq d}(-,M)$ to the the embeddings $(S \hookleftarrow \m^\ell)_{\geq d}$.

Now we reprove that the ideal transform computes the module of twisted global sections (cf., e.g., \cite[\S C.3]{vasconcelos}).
\begin{theorem} \label{thm:Dm}
  The $d$-truncated $\m$-transform $D_{\m,\geq d}$ together with the natural transformation $\eta:\Id_\A \to D_{\m, \geq d}$ is a \nameft{Gabriel} monad of $\A = \Sfpgrmodd$ w.r.t.\ $\C := \Sfpgrmod^0_{\geq d}$.
\end{theorem}
In fact, the theorem holds for all $d \in \Z$. The proof below assumes $d \leq 0$ to avoid case distinctions.
\begin{coro} \label{coro:saturating_power}
  $\Hom_{\geq d}(\m^\ell,M)$ is $\Sfpgrmodd^0$-saturated iff $\ell \geq \delta_{M,d}$.
\end{coro}

Before proving the theorem we state some simple facts about ideal transforms.

\begin{rmrk} \label{rmrk:Dm}
\mbox{}
\begin{enumerate}
  \item \label{rmrk:Dm.a}
    Any $N \in \Sfpgrmod^0$ vanishes in degrees greater than $\reg N$.
    Thus,
    \[
      \GradedHom(L_{\geq \ell}, N)_{\geq \reg N + 1 - \ell} = 0
    \]
    for all $\ell \in \Z$ and $L \in \Sfpgrmod$.
  \item The embedding $M_{\geq t} \hookrightarrow M \in \Sfpgrmodd$ induces (by simple degree considerations) an isomorphism
    \[
      \Hom_{\geq d}(L_{\geq \ell}, M_{\geq t} ) \stackrel{\sim}{\to} \Hom_{\geq d}(L_{\geq \ell}, M )\quad  \mbox{ for all } d \leq t \leq \ell + d \mbox{.}
    \]
    In particular, $D_{\m,\geq d}(M) \cong D_{\m,\geq d}(M_{\geq t})$ for any $t \geq d$ and we are allowed to replace $M$ by any of its truncations.
    \label{rmrk:Dm.>=t}
  \item \label{rmrk:Dm.power}
    For $M \in \Sfpgrmod_{\geq d}$ take $t \geq d$ large enough such that the submodule $M_{\geq t}$ has no $\Sfpgrmodd^0$-torsion.
    Then
    \[
      \Hom_{\geq d}(\m^\ell, M)
      \cong
      \Hom_{\geq d}(\m^\ell, M_{\geq t})
      \cong
      \Hom_{\geq d}(\otimes^\ell \m, M_{\geq t})
    \]
    for all $\ell \geq t - d$ by \eqref{rmrk:Dm.>=t} and Remark~\ref{rmrk:m.tensor}.
    An admissible choice is $t:=\linregd M + 1$, then $\ell \geq t - d \geq \delta_{M,d}$ (cf.~Lemma~\ref{lemm:Dm_colimit}).
    In particular, after replacing $M$ by a high enough truncation we can always assume that $\Hom_{\geq d}(\m^\ell, M) \cong \Hom_{\geq d}(\otimes^\ell \m, M)$.
  \item Since the shift functor $(1): \Sfpgrmodd \to \Sfpgrmod_{\geq d+1},\  M \mapsto M(1),\  \phi \mapsto \phi(1)$ is (quasi-)inverse to the shift functor $(-1): \Sfpgrmod_{\geq d+1} \to \Sfpgrmodd$ and $D_{\m,\geq d}\circ(-1)=(-1)\circ D_{\m,\geq d+1}$ we can restrict the following proofs to $D_{\m,\geq 0}$.
    \label{rmrk:Dm.c}
  \end{enumerate}
\end{rmrk}

\begin{proof}[Proof of Theorem~\ref{thm:Dm}]
We use Theorem~\ref{thm:Csaturating}.
Due to Remark~\ref{rmrk:Dm}.\eqref{rmrk:Dm.c} we only need to consider the case $d=0$.
  \begin{enumerate}
    \item[\ref{thm:Csaturating}.\eqref{thm:Csaturating:ker}] $\C \subset \ker D_{\m,\geq 0}$: \\
      Applying Remark~\ref{rmrk:Dm}.\eqref{rmrk:Dm.a} with $L=S$ (and $L_{\geq L} = S_{\geq \ell} = \m^\ell)$ we conclude that $D_\m$ vanishes\footnote{For $N \in \C$ all modules in the sequential colimit defining $D_{\m, \geq 0}(N)$ vanish for $\ell \geq \delta_{N,0} < \infty$.} on $\Sfpgrmod^0$ and $D_{\m,\geq 0}$ on $\Sfpgrmod_{\geq 0}^0$.
    \item[\ref{thm:Csaturating}.\eqref{thm:Csaturating:im}] $D_{\m,\geq 0}(\A) \subset \Sat_\C(\A)$: \\
      For any $M\in\A$, the map
      \begin{align*}
        \Hom_{\geq 0}\left(S \hookleftarrow \m, D_{\m,\geq 0}(M) \right)
        &=
        \Hom_{\geq 0}\left(S \hookleftarrow \m, \Hom_{\geq 0}(\m^{\delta_{M,0}}, M)\right) \\
        &=
        \Hom_{\geq 0}\left(S \hookleftarrow \m, \Hom_{\geq 0}(\otimes^{\delta_{M,0}} \m, M)\right) \\
        &=
        \Hom_{\geq 0}\left(\otimes^{\delta_{M,0}} \m \hookleftarrow \otimes^{\delta_{M,0}+1} \m, M\right) \\
        &=
        \Hom_{\geq 0}\left(\otimes^{\delta_{M,0}} \m, M \right)  \to \Hom_{\geq 0}\left( \otimes^{\delta_{M,0}+1} \m, M\right) \\
        &=
        \Hom_{\geq 0}\left(\m^{\delta_{M,0}}, M\right)  \to \Hom_{\geq 0}\left(\m^{\delta_{M,0}+1}, M\right)
      \end{align*}
      is an isomorphism by Lemma~\ref{lemm:Dm_colimit} proving statement~\eqref{coro:saturation_0.S-m} of Corollary~\ref{coro:saturation_0}.
      We have repeatedly used Remark~\ref{rmrk:Dm}.\eqref{rmrk:Dm.power} and the adjunction between $\otimes$ and $\Hom_{\geq 0}$.
    \item[\ref{thm:Csaturating}.\eqref{thm:Csaturating:exact}] $G := \cores_{\Sat_\C(\A)} D_{\m,\geq 0}$ is exact: \\
      Applying $\Hom_{\geq 0}(\m^\ell,-)$ to the short exact sequence $L \hookrightarrow M \twoheadrightarrow N$ in $\Sfpgrmod_{\geq 0}$ yields the exact sequence
      \[
        \Hom_{\geq 0}(\m^\ell,L) \hookrightarrow \Hom_{\geq 0}(\m^\ell,M) \to \Hom_{\geq 0}(\m^\ell,N) \to \Ext_{\geq 0}^1(\m^\ell,L)
      \]
      as part of the long exact covariant $\Ext_{\geq 0}$-sequence.
      Since $\Ext_{\geq 0}^1(\m^\ell,L)$ is quasi-zero by Remark~\ref{rmrk:TorExtQZ}.\eqref{rmrk:TorExtQZ.Extofm} the sequence is exact up to defects in $\Sfpgrmod_{\geq 0}^0$.
    \item[\ref{thm:Csaturating}.\eqref{thm:Csaturating:idem}] $\eta D_{\m,\geq 0} = D_{\m,\geq 0} \eta$: \\
    We repeatedly use the adjunction between $\otimes$ and $\Hom_{\geq 0}$ and Lemma~\ref{lemm:Dm_colimit} to interchange the involved sequential colimits over $\ell'$ and $\ell''$ by a common $\ell \geq \ell',\ell''$, high enough to stabilize both colimits:
      \begin{align*}
        \eta_{D_{\m,\geq 0}(M)}
        &=
        \varinjlim_{\ell'} \Hom_{\geq 0}(S \hookleftarrow \m^{\ell'}, \varinjlim_{\ell''} \Hom_{\geq 0}(\m^{\ell''}, M))
        \\
        &=
        \Hom_{\geq 0}(S \hookleftarrow \m^\ell, \Hom_{\geq 0}(\m^\ell, M))
        \\
        &=
        \Hom_{\geq 0}((S \hookleftarrow \m^\ell) \otimes_S \m^\ell, M)
        \\
        &=
        \Hom_{\geq 0}(\m^\ell, \Hom_{\geq 0}(S \hookleftarrow \m^\ell, M))
        \\
        &=
        \varinjlim_{\ell''} \Hom_{\geq 0}(\m^{\ell''}, \varinjlim_{\ell'} \Hom_{\geq 0}(S \hookleftarrow \m^{\ell'}, M))
        \\
        &=
        D_{\m,\geq 0}(\eta_M)
        \mbox{.}
      \end{align*}
      The proof implicitly uses commuting diagrams of morphisms in $\Sfpgrmod_{\geq 0}$ to justify the equality signs.\footnote{%
	We could have used the fact that $D_{\m,\geq 0}=\varinjlim_\ell \Hom_{\geq 0}(\m^{\ell},-)$ commutes with directed colimits.
	However, the general form of the second statement is not quite trivial \cite[Coro.~3.4.11]{BrSh} (the directed colimit is called direct limit in \cite[Terminology 3.4.1]{BrSh}).
	Note that although the ideal transform commutes with \emph{directed} colimits, it does not generally commute with arbitrary finite colimits, for otherwise it would be right exact and hence exact.
      }
    \item[\ref{thm:Csaturating}.\eqref{thm:Csaturating:eta}] $\eta \iota$ is a natural isomorphism: \\
      Let $M \in \Sfpgrmod_{\geq 0}$ be saturated w.r.t.\  $\Sfpgrmod_{\geq 0}^0$.
      Applying $\Hom_{\geq 0}(-,M)$ to the short exact sequence $S/\m^\ell \twoheadleftarrow S \hookleftarrow \m^\ell$ yields
      \[
        \underbrace{\Hom_{\geq 0}(S/\m^\ell,M)}_{\cong 0} \hookrightarrow M \xrightarrow{\eta^\ell_M} \Hom_{\geq 0}(\m^\ell,M) \twoheadrightarrow \underbrace{\Ext_{\geq 0}^1(S/\m^\ell,M)}_{\cong 0}
      \]
      since $S/\m^\ell \in \Sfpgrmod_{\geq 0}^0$. In other words, $\eta^\ell_M$ is an isomorphism for all $\ell$.
      \qedhere
  \end{enumerate}
\end{proof}

\begin{rmrk}
  The saturation process of $M \in \Sfpgrmod$ conducted by $D_\m$ brings $\linreg$ to $-\infty$, whereas $\reg$ is only brought down to the regularity of the sheafification.
\end{rmrk}

Since the \nameft{Frobenius} powers $\m^{[\ell]} := \langle x_0^\ell, \ldots, x_n^\ell \rangle$ satisfy $\m^\ell \geq \m^{[\ell]} \geq \m^{(n+1) \ell}$ we can use them instead of $\m^\ell$ them in the above sequential colimits.
They are computationally superior since their number of generators does not increase with $\ell$.
In other words, the module $\Hom_{\geq d}(\m^{[\delta_{M,d}]}, M)$ is $\Sfpgrmodd^0$-saturated.
Alternatively, one could iteratively ($\delta_{M,d}$ times) apply $\Hom_{\geq d}(\m,-)$ to (the $\Sfpgrmodd^0$-torsion-free factor of) $M$.
It depends on the example which algorithm is faster.

\section{Graded \texorpdfstring{$S$}{S}-modules and linear \texorpdfstring{$E$}{E}-complexes}\label{sec:Emodules}

In this section we describe how to translate the module structure of $M \in \Sfpgrmod$ into the structure of a linear complex $\RR(M)$ over the exterior algebra $E := \wedge V$, which is \nameft{Koszul} dual to $S=\Sym V^*$.
This translation turns out to be functorial, algorithmic, and an adjoint equivalence of categories.
We denote the category of finitely generated graded $E$-modules by $\Efpgrmod$.

Let $e_0,\ldots,e_n$ denote a $B$-basis $V$ of which the indeterminates $x_0,\ldots,x_n$ of $S$ form the dual $B$-basis of $W = V^* = \Hom_B(V,B)$.
We set $\deg(e_i)=-1$ for all $i=0,\ldots,n$.

\subsection{The functor \texorpdfstring{$\RR$}{R}}

The $B$-linear maps
\[
  \mu^i(x_j): M_i \to M_{i+1}, m \mapsto x_j m, \quad \mbox{for } j=0,\ldots,n, \mbox{ and } i \in \Z
\]
induced by the indeterminates $x_j$ encode the graded $S$-module structure of an $M \in \Sfpgrmod$ (cf.~Algorithm~\ref{algo:degree_i} for an algorithm to compute $M_i$).

\begin{exmp} \label{exmp:R_str_sheaf}
For $S:=B[x_0,x_1]$ consider the free $S$-module $M:=S=S(0)$ of rank $1$.
Each graded part $M_i$ is a free $B$-module for which we fix a basis of monomials, e.g., $M_0 = \langle 1 \rangle_B$, $M_1 = \langle x_0, x_1 \rangle_B$, $M_2 = \langle x_0^2, x_0 x_1, x_1^2 \rangle_B$.
Then the matrices
\begin{align*}
  0: && \mu^0(x_0) = \left(\begin{smallmatrix} 1&0 \end{smallmatrix}\right)\text{, } & \mu^0(x_1) = \left(\begin{smallmatrix} 0&1 \end{smallmatrix}\right), \\
  1: && \mu^1(x_0) = \left(\begin{smallmatrix} 1&0&0 \\ 0&1&0\end{smallmatrix}\right)\text{, } & \mu^1(x_1) = \left(\begin{smallmatrix} 0&1&0 \\ 0&0&1 \end{smallmatrix}\right), \\
  \vdots\phantom{:}
\end{align*}
represent the maps $\mu^i(x_j)$.
\end{exmp}

Using the $B$-basis $(e_0,\ldots,e_n)$ of $V$ define for each $i \in \Z$ the map $\mu^i$ as the composition
\[
  \mu^i:
  \left\{
  \begin{array}{rcccl}
    M_i &\to& \End_B(V) \otimes_B M_i &\to& V \otimes_B M_{i+1}, \\
    m &\mapsto& \id_V \otimes m &\mapsto& \sum_{i=0}^n e_j \otimes x_j m
  \end{array}
  \right.\mbox{.}
\]

Using the natural isomorphism $\Hom_B(M_i, V \otimes_B M_{i+1}) \cong \Hom_\Efpgrmod(E \otimes_B M_i, E \otimes_B M_{i+1})$ each $\mu^i$ can equally be understood as a map of \emph{graded} $E$-modules
\[
  \mu^i: E \otimes_B M_i \to E \otimes_B M_{i+1} \mbox{,}
\]
where the $B$-module $M_j$ is considered as a graded $B$-module concentrated in degree $j$ and, therefore, $E \otimes_B M_j$ is generated by (a generating set of) $M_j$ in degree $j$.

For a better functorial behavior we replace $E$ by its $B$-dual \cite[§16C]{lam99}
\[
  \omega_E := \Hom_B(E,B) \cong \wedge W \cong \wedge^{n+1} W \otimes_B E
\]
in the above maps.\footnote{It is again a free graded $E$-module which is \emph{non}naturally isomorphic to $E(-n-1)$.}
In particular, $\omega_E$ lives in the degree interval $0,\ldots,n+1$ and its socle $(\omega_E)_0$, which is naturally isomorphic to $B$, is concentrated in degree $0$.
We denote the distinguished generator of the socle corresponding to $1_B$ by $1_{\omega_E}$.

This change of language is justified by reinterpreting $\mu^i: M_i \to V \otimes_B M_{i+1}$ as a map $\mu^i:W \otimes_B M_i \to M_{i+1}$ using the adjunction
\[
  \Hom_B(W \otimes_B X, Y) \cong \Hom_B(X,\Hom_B(W,Y)) \cong \Hom_B(X, W^* \otimes_B Y) \mbox{.}
\]
The graded $E$-module $\omega_E \otimes_B M_j$ has (compared with $E \otimes_B M_j$) the advantage of having the $B$-module $M_j$ as its socle interpreted as a graded $E$-module concentrated in degree $j$.

The commutativity of $S$ implies that the composed map $\mu^{i+1} \circ \mu^i: \omega_E \otimes_B M_i \to \omega_E \otimes_B M_{i+2}$ is zero.
Thus, the sequence of $\mu_i$'s yields the so-called \textbf{$\RR$-complex} (cf.~\cite[§2]{EFS} and \cite[§2]{ES})
\[
  \RR(M): 
  \cdots \to \omega_E \otimes_B M_i \xrightarrow{\mu^i} \omega_E \otimes_B M_{i+1} \xrightarrow{\mu^{i+1}} \omega_E \otimes_B M_{i+2} \to \cdots 
\]

\begin{exmp_contd}[\ref{exmp:R_str_sheaf}]
For $M=S(0)$ we obtain the $\RR$-complex
\begin{center}
{\small
\begin{tikzpicture}[scale=0.9]
  \draw[-stealth'] (-2.75,0) node[left] {$0$} -- (-2.25,0) node[right] {$\omega_E(0)^1$};
  \draw[-stealth'] (-0.65,0) -- node[above] {$\left(\begin{smallmatrix} e_0 & e_1 \end{smallmatrix}\right)$} (0.5,0) node[right] {$\omega_E(-1)^2$};
  \draw[-stealth'] (2.5,0) -- node[above] {$\left(\begin{smallmatrix} e_0 & e_1 & \cdot \\ \cdot & e_0 & e_1 \end{smallmatrix}\right)$} (4.25,0) node[right] {$\omega_E(-2)^3$};
  \draw[-stealth'] (6.25,0) -- node[above] {$\left(\begin{smallmatrix} e_0 & e_1 & \cdot & \cdot \\ \cdot & e_0 & e_1 & \cdot \\ \cdot & \cdot & e_0 & e_1 \end{smallmatrix}\right)$} (8.75,0) node[right] {$\omega_E(-3)^4$};
  \draw[-stealth'] (10.75,0) -- (11.5,0) node[right] {$\cdots$};
\end{tikzpicture}
}
\end{center}
\end{exmp_contd}

\begin{lemm}[{\cite[Prop.~2.3]{EFS}}] \label{lemm:Betti}
There exists a natural isomorphism
\[
  H^a(\RR(M))_{a + i} \cong \Tor^S_i(B,M)_{a + i} \mbox{.}
\]
\end{lemm}
\begin{proof}
  The idea is to interpret the \emph{bi}graded differential module $\omega_E \otimes_B M$ either as $\RR(M)$ or as the \nameft{Koszul} resolution of $B$ tensored with $M$ over $S$.
\end{proof}

Lemmas~\ref{lemm:Betti} and \ref{lemm:Betti=Bass} imply the following lemma, an important technical insight for the rest of this paper.

\begin{lemm}[Key Lemma] \label{lma:HRR_Ext}
  There exists a natural isomorphism
  \[
    H^a(\RR(M))_{a + i} \cong \Ext_\bullet^{n+1-i}(\wedge^{n+1}V,M)_{a + i}\mbox{.}
  \]
  Hence, there is a noncanonical isomorphism $H^a(\RR(\grM))_{a + n + 1 - j} \cong \Ext_\bullet^j(B,\grM)_{a - j}$.
\end{lemm}

Let $A$ be either $S$ or $E$.
An epimorphism in $\Afpgrmod$ is said to be $B$-split if it splits as a morphism over $B$.
A graded module $P \in \Afpgrmod$ is said to be \textbf{relatively projective} (with respect to $B$) if $\Hom_S(P,-)$ sends $B$-split epis to surjections.
Any module of the form $A\otimes_B M$, where $M$ is a $B$-module, is called \textbf{relatively free} (with respect to $B$).
By \cite[Proposition~1.1]{ES}, an $N \in \Afpgrmod$ is relatively projective if and only if it is relatively free.

A complex $C=C^\bullet$ of graded $E$-modules is called \textbf{linear} if each $C^i$ is relatively free (with respect to $B$) with socle concentrated in degree $i$.\footnote{and hence $C^i$ is generated in degree $i+n+1$.}
The \textbf{regularity} of a \emph{linear} complex $C$ is defined as
\[
  \reg C := \sup \{ a \in \Z \mid H^a(C) \neq 0 \} \in \Z \cup \{ -\infty, \infty \} \mbox{.}
\]

Lemma~\ref{lemm:Betti} or \ref{lma:HRR_Ext} connects the regularity of a graded module with that of its $\RR$-complex.

\begin{coro} \label{coro:reg}
  For $M \in \Sfpgrmod$ the equality $\reg M = \reg \RR(M)$ holds.
\end{coro}

These definitions allow us to describe the image of $\RR$.

\begin{defn}\label{defn:linear_complex}
  We denote by $\Egrlin$ the full subcategory of complexes $C$ of graded $E$-modules satisfying
  \begin{enumerate}
    \item $C$ is linear;
    \item each $C^i$ is finitely generated;
    \item $C$ is left bounded;
    \item $\reg C < \infty$.
  \end{enumerate}
  By $\Egrlinzero$ we denote the thick subcategory of bounded complexes.
  
  Finally, for any $d \in \Z$, denote by $\Egrlind$ the full subcategory of complexes in $\Egrlin$ with $C^{<d} = 0$ and by $\Egrlindzero := \Egrlind \cap \Egrlinzero$.
\end{defn}

\begin{rmrk}\label{rmrk:data_structure_egrlind}
  An object $C\in\Egrlind$ can be represented on a computer by the finite complex $C^d\to C^{d+1}\to\ldots\to C^{j-1}\to C^j$ provided that $j>\reg C$.
  The part $C^{>j}$ of $C$ can be recovered by an injective resolution of $\coker(C^{j-1}\to C^j)$.
  This relatively injective resolution is isomorphic to $\GradedHom(-,E)$ applied to a relatively projective resolution of $\GradedHom(\coker(C^{j-1}\to C^j),E)$.
  A morphisms in $\Egrlin$ can be represented on the computer by a chain morphism between two such finite complexes, one only needs to extend these complexes to equal cohomological degrees.
  Again, the part of the morphism in higher cohomological degrees can be computed by an injective resolution.
\end{rmrk}

\begin{prop}
  The construction $\RR$ induces two fully faithful functors $\RR: \Sfpgrmod\to\Egrlin$ and $\RRd:\Sfpgrmodd\to\Egrlind$ for all $d\in\Z$.
\end{prop}
\begin{proof}
  As $M\in\Sfpgrmod$ is finitely generated, $\RR(M)$ is left bounded.
  By definition, each $\RR(M)^i=\omega_E\otimes_B M_i$ is a finitely generated graded relatively free module with socle $M_i$ concentrated in degree $i$.
  Furthermore $\reg \RR(M) = \reg M < \infty$ by Corollary~\ref{coro:reg}.

  A graded morphism $\phi:M\to N$ induces morphisms $\phi_i:M_i\to N_i$ for all $i\in\Z$.
  Tensoring with $\omega_E$ yields morphisms $\RR(\phi)^i:\RR(M)^i\to\RR(N)^i$.
  These morphisms are chain morphisms, as $x_j\circ\phi_i=\phi_{i+1}\circ x_j$ and the $\mu^i$ are induced by the $x_j$ for all $i\in\Z$ and all $0\le j\le n$.
  
  Restricting $\RR$ to $\Sfpgrmodd$ corestricts to $\Egrlind$ by construction.
  These functors are obviously faithful.
  The fullness $\RR$ and $\RRd$ follows directly from the below Proposition~\ref{prop:from_linear_eventually_exact_to_module} and Corollary~\ref{coro:from_linear_eventually_exact_to_module_ged}, respectively.
\end{proof}

\subsection{The functor \texorpdfstring{$\RR$}{R} induces an equivalence}

The functor $\RR$ is an equivalence $\Sfpgrmod\xrightarrow{\sim}\Egrlin$ by \cite[Prop.~2.1]{EFS}.
In this section we explicitly construct the left adjoint quasi-inverse $\MM$ of $\RR$ and thus show \emph{constructively} that $\RR$ is an adjoint equivalence.

\begin{prop} \label{prop:from_linear_eventually_exact_to_module}
  There exists a functor $\MM:\Egrlin\to\Sfpgrmod$ such that $\MM\dashv\RR$ is an adjoint equivalence of categories which sends $\Sfpgrmod^0$ to $\Egrlinzero$.
\end{prop}
\begin{proof}
  Let $(C,\mu)\in\Egrlin$.

  For a preparatory step, assume that
  \begin{align} \tag{A} \label{A}
    \mbox{$H^r(C)$ is the only nonvanishing cohomology (this implies that $C^{<r} = 0$).}
  \end{align}
  Consider $\mu^r:W\otimes_B C^r_r\to C^{r+1}_{r+1}$ and extend $\ker(\mu^r)\xrightarrow{\kappa}W\otimes_B C^r_r$ to a map $S\otimes_B \ker(\mu^r)\xrightarrow{}S\otimes_B C^r_r$.
  Define $\MM(C)$ as its cokernel (with relatively free presentation $\pi_r: S \otimes_B C^r_r \twoheadrightarrow \MM(C)$).

  To justify the correctness of this preparatory step let $M\in\Sfpgrmod$ with $M_{<r}=0$ and such that $\RR(M)$ satisfies assumption~\eqref{A}.
  The natural isomorphism $\widetilde{\delta}^{-1}:M_r\xrightarrow{\sim}\MM(\RR(M))_r:m\mapsto 1_S\otimes_B 1_{\omega_E} \otimes_B m$ identifies a minimal set of generators $M$ with one of $\MM(\RR(M))$.
  The assumption~\eqref{A} for $\RR(M)$ is equivalent, by Lemma~\ref{lemm:Betti}, to $M$ being generated in degree $r$ and having a relatively free resolution which is linear in the $x_i$.
  In particular, the only relations involving the indeterminates $x_i$ of the finite set of generators of $M$ in $M_r$ are linear relations.
  All these linear relations are encoded in the map $\RR(M)^r\to\RR(M)^{r+1}$.
  The construction of $\MM$ above just imposes these linear relations of the generators of $M$ to the generators of $\MM(\RR(M))$.
  In particular, $\widetilde{\delta}$ induces an isomorphism $\delta_M:\MM(\RR(M))\to M$.
  Similarly, there exists an isomorphism $\eta_C:C\to\RR(\MM(C))$ for any $C\in\Egrlin$ satisfying assumption~\eqref{A}.

  For a general $(C,\mu)\in\Egrlin$, there is a bound $r$ (e.g., any $r>\reg(C)$) such that the preparatory step applies to $(C^{\ge r},\mu^{\ge r})$.
  Then, we inductively define $\MM(C)$ by decreasing the cohomological degree $d$.
  Let $(C,\mu)\in\Egrlin$ be a complex and $d < r$ such that $\MM(C^{\ge d+1})$ is defined by the induction hypothesis with relatively free presentation
  $\pi_{d+1}: S\otimes_B(C^{d+1}_{d+1}\oplus\ldots\oplus C^r_r)\twoheadrightarrow\MM(C^{\ge d+1})$.
  We define $\MM(C^{\ge d})$ as a pushout of the span of $\beta$ and $\gamma$ defined as follows:
  Let $\alpha: S\otimes_B W\to S:p\otimes x_i\to x_ip$ and $\iota:C^{d+1}_{d+1}\hookrightarrow C^{d+1}_{d+1}\oplus\ldots\oplus C^r_r$ be the embedding in the direct sum.
  Now set $\beta := \alpha\otimes_B C^d_d$ and $\gamma := \pi_{d+1} \circ (S\otimes_B(\iota\circ\mu^d))$ with common source $S \otimes_B W \otimes C^d_d$ (recall, $\mu^d:W\otimes_B C^d_d\to C^{d+1}_{d+1}$).
  This inductive step of the construction of $\MM$ is the reverse construction of $\RR$.

  To apply $\MM$ to a morphism $\phi:C\to D$ in $\Egrlin$ we use the identification of $\MM(C)_i$ with $C_i^i$, map $C_i^i$ using $\phi^i$ to $D_i^i$, which we finally identify with $\MM(D)_i$.
  
  This equivalence of categories is an adjoint equivalence.
  We already have constructed the unit $\eta$ and counit $\delta$ as natural isomorphisms in the preparatory step.
  This unit and counit naturally extends into lower cohomological degrees using the natural $B$-isomorphisms $C^i_i\xrightarrow\sim\MM(C)_i:c\mapsto 1_S\otimes_B c$ and $M_i\xrightarrow\sim \RR(M)^i_i:m\mapsto 1_{\omega_E}\otimes_B m$.
  The triangle identities are easily verified.
\end{proof}

\begin{coro}\label{coro:from_linear_eventually_exact_to_module_ged}
  The restriction-corestriction $\MMd:\Egrlind\to\Sfpgrmodd$ of $\MM$ and the functor $\RRd$ form an adjoint equivalence $\MMd\dashv\RRd$, which sends $\Sfpgrmodd^0$ to $\Egrlindzero$.
\end{coro}

\subsection{Saturated linear complexes}

We now give a characterization of saturated linear complexes corresponding to the one we gave for graded modules.

\begin{defn}
  The \textbf{linear regularity} of a linear complex $C \in \Egrlin$ is defined as
  \[
    \linreg C := \max \{ a \in \Z \mid H^a(C)_{a+n+1} \neq 0 \mbox{ or } H^a(C)_{a+n} \neq 0 \} \in \Z \cup \{-\infty\} \mbox{.}
  \]
\end{defn}

We get a further characterization of $\Egrlinzero$-saturated linear complexes.

\begin{coro}\label{coro:linear_sat}
  A complex $C\in\Egrlin$ is $\Egrlinzero$-saturated iff $\linreg C = -\infty$.
\end{coro}
\begin{proof}
  By Proposition~\ref{prop:saturation}, the module $\MM(C)$ is $\Sfpgrmod^0$-saturated if $\Ext_\bullet^j(B,\MM(C)) = 0$ for $j\in\{0,1\}$.
  This is equivalent to $H^a(\RR(\MM(C)))_{a+n+1-j}=0$ for $j\in\{0,1\}$ by the key Lemma~\ref{lma:HRR_Ext}.
  The claim follows from $C\cong\RR(\MM(C))$.
\end{proof}

The key Lemma~\ref{lma:HRR_Ext} also implies:
\begin{coro} \label{coro:linreg}
  $\linreg C = \linreg \MM(C)$ for all $C \in \Egrlin$.
\end{coro}

The localizing subcategory $\Sfpgrmodd^0$ of $\Sfpgrmodd$ corresponds via the adjoint equivalence $\MM \dashv \RR$ to the full localizing subcategory $\Egrlindzero$ of right bounded complexes in $\Egrlind$, i.e., of those complexes $C\in\Egrlind$ with $C^{\ge \ell}=0$ for $\ell$ large enough.
A module $M \in \Sfpgrmodd$ is then $\Sfpgrmodd^0$-saturated if and only if $\RR(M)$ is $\Egrlindzero$-saturated, i.e., the adjoint equivalence $\MMd \dashv (\RRd: \Sfpgrmodd\to\Egrlind)$ restricts to an adjoint equivalence between the full subcategories of $\Sfpgrmodd^0$-saturated resp.\ $\Egrlindzero$-saturated objects.

The definition of the linear regularity of complexes in $\Egrlind$ and the characterization of $\Egrlindzero$-saturated linear complexes is a little bit more subtle and is therefore deferred to the next section.
The reason is that the lowest cohomology $H^d(C)$ has to be treated separately.

\section{Saturation of linear complexes}\label{sec:saturate_linear_complex}

The ideal transform in Section~\ref{sect:ideal_transforms} leads to an algorithm for the saturation of graded $S$-modules.
In this section, we present an algorithm to saturate linear complexes.
The adjoint equivalence $\MM\dashv\RR$ translates this to a second algorithm for the saturation of graded $S$-modules.

Corollary~\ref{coro:linear_sat} indicates that one has to modify a linear complex $C$ until $H^a(C)_{a+n+1}=0$ and $H^a(C)_{a+n}=0$.
Our purely linear saturation is similar to that of the \nameft{Tate} resolution in that one truncates $C$ in cohomological degree high enough and then computes a suitable part in lower cohomological degrees.
In contrast to the \nameft{Tate} resolution, our approach remains in the category of linear complexes, as we do not take free presentations of kernels to compute the part of lower cohomological degree, but so-called purely linear kernels.
We can also truncate above the linear regularity, a lower bound of the \nameft{Castelnuovo-Mumford} regularity.
For the relation between the purely linear saturation and the \nameft{Tate} resolution see Remark~\ref{rmrk:E1}.

\subsection{Purely linear kernels}

Let $C^i,C^{i+1}\in \Efpgrmod$ be relatively free with socle concentrated in degree $i$ and $i+1$, respectively\footnote{or, equivalently, freely generated in degree $i+n+1$ and $i+n+2$, respectively.}.
We call a morphism $\phi^i:C^i\to C^{i+1}$ \textbf{purely linear (of degree $i$)} if its kernel vanishes in the top degree $i+n+1$.

\begin{wrapfigure}[6]{r}{5.4cm}
  \centering
  \vspace{-1.0em}
  \begin{tikzpicture}[label/.style={postaction={
      decorate,
      decoration={markings, mark=at position .5 with \node #1;}}}]
      \coordinate (r) at (2,0);
      \coordinate (u) at (0,1.5);
      \node (Mi) {$C^i$};
      \node (Mi1) at ($(Mi)+(r)$) {$C^{i+1}$};
      \node (K) at ($(Mi)-(r)$) {$K^{i-1}$};
      \node (Mm1) at ($(K)-(u)$) {$L^{i-1}$};
      \draw[-stealth'] (Mi) -- node[above]{$\phi^i$} (Mi1);
      \draw[-stealth'] (K) -- node[above]{$\kappa$} (Mi);
      \draw[-stealth'] (Mm1) -- node[right]{$\lambda$} (Mi);
      \draw[bend right,-stealth'] (Mm1) to node[above]{$0$} (Mi1);
      \draw[-stealth',dotted,thick] (Mm1) -- node[left]{$\psi$} (K);
    \end{tikzpicture}
\end{wrapfigure}
\mbox{}
\vspace{-1em}
\begin{defn}
  Let $\phi^i:C^i\to C^{i+1}$ be purely linear of degree $i$.
  A purely linear morphism $\kappa: K^{i-1}\to C^i$ of degree $i-1$ with $\phi^i\circ\kappa=0$ is called \textbf{purely linear kernel} if for any purely linear $\lambda:L^{i-1}\to C^i$ of degree $i-1$ with $\phi^i\circ\lambda=0$ there exists a unique morphism $\psi:L^{i-1}\to K^{i-1}$ with $\kappa\circ\psi=\lambda$.
\end{defn}

\begin{lemm}\label{lemm:exist_purely_linear_kernel}
  Each purely linear morphism has a purely linear kernel, which, by the universal property, is unique up to a unique isomorphism.
\end{lemm}
\begin{proof}
  We denote the restriction of any morphism $\beta$ to the graded part of degree $i+n$ by $\beta_{i+n}$.

  Let $\phi^i:M^i\to M^{i+1}$ be purely linear of degree $i$ and $\nu: N^{i-1}\hookrightarrow M^i$ be its kernel.
  Denote by $K^{i-1}:=N^{i-1}_{i+n}\otimes_B E$ and by $\lambda:K^{i-1}\to N^{i-1}$ the map induced by the identity on $N^{i-1}_{i+n}$.
  We show that $\kappa:=\nu\circ\lambda:K^{i-1}\to M^i$ is a purely linear kernel of $\phi^i$.
  
  By definition, $K^{i-1}$ and $M^i$ are relatively free generated in degree $i+n$ and degree $i+n+1$, respectively.
  As $\phi^i$ is purely linear, $N^{i-1}$ lives in the degree interval $i,\ldots,i+n$.
  In particular, $\nu_{i+n}$ is a kernel of $\phi^i_{i+n}$.
  By definition, $\lambda_{i+n}:K^{i-1}_{i+n}\to N^{i-1}_{i+n}$ is an isomorphism and thus also $\kappa_{i+n}$ is a kernel of $\phi^i_{i+n}$.
  In particular, the kernel of $\kappa$ lives in the degree interval $i-1,\ldots,i+n-1$.
  Thus, $\kappa$ is purely linear of degree $i-1$.
  
  The composition $\phi^i\circ\nu$ is zero and $\kappa$ factors over $\nu$ by construction.
  Thus, $\phi^i\circ\kappa=0$.

  \begin{center}
    \begin{tikzpicture}[label/.style={postaction={
      decorate,
      decoration={markings, mark=at position .5 with \node #1;}}}]
      \coordinate (r) at (2,0);
      \coordinate (u) at (0,1.3);
      \node (Mi) {$M^i$};
      \node (Mi1) at ($(Mi)+(r)$) {$M^{i+1}$};
      \node (N) at ($(Mi)-(r)+(u)$) {$N^{i-1}$};
      \node (K) at ($(Mi)-(r)$) {$K^{i-1}$};
      \node (Mm1) at ($(Mi)-(u)-(r)$) {$M^{i-1}$};
      \draw[-stealth'] (Mi) to node[above]{$\phi^i$} (Mi1);
      \draw[right hook-stealth'] (N) to node[above right]{$\nu$} (Mi);
      \draw[-stealth'] (K) to node[left]{$\lambda$} (N);
      \draw[-stealth'] (Mm1) to node[right]{$\phi^{i-1}$} (Mi);
      \draw[-stealth'] (K) to node[above left]{$\kappa$} (Mi);
      \draw[-stealth',dotted,thick] (Mm1) to node[left]{$\psi$} (K);
      \draw[bend right,-stealth'] (Mm1) to node[above]{$0$} (Mi1);
    \end{tikzpicture}
  \end{center}

  To show the universal property of $\kappa$ let $\phi^{i-1}:M^{i-1}\to M^i$ be purely linear with $\phi^i\circ\phi^{i-1}=0$.
  From the universal property of $\kappa_{i+n}$ as a kernel, there exists a unique $\psi_{i+n}:M^{i-1}_{i+n}\to K^{i-1}_{i+n}$ with $\kappa_{i+n}\circ\psi_{i+n}=\phi^{i-1}_{i+n}$, since $\phi^{i-1}_{i+n}\circ\phi^i_{i+n}=0$.
  We define
  \[
    \psi:=\psi_{i+n}\otimes_B E:M^{i-1}\cong M^{i-1}_{i+n}\otimes_B E\longrightarrow K^{i-1}\cong K^{i-1}_{i+n}\otimes_B E\mbox{,}
  \]
  which extends $\psi_{i+n}$ to a morphism of graded $E$-modules.
  Finally, $\phi^{i-1}=\kappa\circ\psi$ since $\kappa_{i+n}\circ\psi_{i+n}=\phi^{i-1}_{i+n}$ and $\phi^{i-1}$ is uniquely determined by $\phi^{i-1}_{i+n}$.
\end{proof}

Note that all steps in the proof of this last lemma are constructive.

We can now state the definition of linear regularity for complexes in $\Egrlind$.
\begin{defn}
  Define for any $d \leq 0$ the \textbf{$d$-th truncated linear regularity} $\linregd C \in \Z_{\geq d} \cup \{ -\infty \}$ of $C \in \Egrlind$ as follows: \\
  If there exists an $a \in \Z_{> d}$ such that $H^a(C)_{a+n+1} \neq 0 \mbox{ or } H^a(C)_{a+n} \neq 0$ then
  \[
    \linregd C := \max \{ a \in \Z_{> d} \mid H^a(C)_{a+n+1} \neq 0 \mbox{ or } H^a(C)_{a+n} \neq 0 \} \in \Z_{>d} \mbox{.}
  \]
  Otherwise, if the lowest morphism $C^d \to C^{d+1}$ is a purely linear kernel (of $C^{d+1} \to C^{d+2}$) then $\linregd C := -\infty$ else $\linregd C := d$.
\end{defn}

\begin{coro} \label{coro:linear_satd}
  $C\in\Egrlind$ is $\Egrlindzero$-saturated iff $\linregd C = -\infty$.
\end{coro}
\begin{proof}
  The claim follows from Corollary~\ref{coro:linear_sat} and Corollary~\ref{coro:saturation_0} (by Remark~\ref{rmrk:Dm}.\eqref{rmrk:Dm.c} we only need to consider the case $d=0$).
\end{proof}

\begin{coro} \label{coro:linregd}
  $\linregd C = \linregd \MMd(C)$ for all $C \in \Egrlind$.
\end{coro}

\begin{prop}\label{prop:saturation_and_purely_linear_kernels}
  A $(C,\mu)\in\Egrlind$, is $\Egrlindzero$-saturated if and only if $\mu^i$ is the purely linear kernel of $\mu^{i+1}$ for all $i\ge d$.
\end{prop}
\begin{proof}
  By the proof of Lemma~\ref{lemm:exist_purely_linear_kernel}, a morphism $\kappa: K^{i-1}\to M^i$ is a purely linear kernel of a purely linear morphism $\phi^i:M^i\to M^{i+1}$ of degree $i$ if and only if it is purely linear and $K^{i-1}\xrightarrow\kappa M^i \xrightarrow{\phi^i} M^{i+1}$ is a complex which is exact in degrees $i+n$ and $i+n+1$.
  Now, the claim follows from the characterizations of saturated linear complexes in Corollary~\ref{coro:linear_satd}.
\end{proof}

\subsection{Saturation of a linear complex}

In this subsection we algorithmically saturate linear complexes by iteratively computing purely linear kernels.

Let $(C,\mu)\in\Egrlind$ with regularity $r\in\Z$.
Define the \textbf{purely linear saturation (truncated in degree $d$)} functor $\bSd:\Egrlind\to\Egrlind$ as follows.
The idea is to truncate the complex above the \emph{linear} regularity and then to ``saturate'' it recursively by purely linear kernels, more precisely:
For cohomological degrees greater than the \emph{linear} regularity $r = \linregd C$ define $\bS^{\ge r+1}$ by setting $\bS^{\ge r+1}(C,\mu):=(C^{\ge r+1},\mu^{\ge r+1})$. 
Assume that $(D^{\ge i},\tau^{\ge i})=\bS^{\ge i}(C,\mu)$ is defined for some $i>d$.
Let $\tau^{i-1}:D^{i-1}\to D^i$ be the purely linear kernel of $\tau^i$.
Define $\bS^{\ge i-1}(C,\mu)$ by adding $\tau^{i-1}$ to $(D^{\ge i},\tau^{\ge i})$ in cohomological degree $i-1$.

The morphism part $\bSd(\varphi)$ for $\varphi:(C,\mu_C)\to(C',\mu_{C'})$ in $\Egrlind$ is induced by the identity in high degrees.
The universal property of the purely linear kernels implies a \emph{unique} completion of the square
\begin{center}
\begin{tikzpicture}
  \coordinate (right) at (3,0);
  \coordinate (above) at (0,1.5);
  \coordinate (rightabove) at ($(right)+(above)$);
  
  \node (00)                 {$\bSd(C')^\ell$};
  \node (10) at (right)      {$\bSd(C')^{\ell+1}$};
  \node (01) at (above)      {$\bSd(C)^\ell$};
  \node (11) at (rightabove) {$\bSd(C)^{\ell+1}$};
  
  \draw[-stealth'] (11) -- node[right] {$\bSd(\varphi)^{\ell+1}$} (10);
  \draw[-stealth',dashed] (01) -- (00);
  \draw[-stealth'] (00) -- (10);
  \draw[-stealth'] (01) -- (11);
  
\end{tikzpicture}
\end{center}
and thus iteratively constructs the chain morphisms in lower degrees.

\begin{theorem} \label{thm:tate_monad}
  Let $\A = \Egrlind$ and $\C := \Egrlindzero$.
  There exists a natural transformation $\eta:\Id_\A \to \bSd$ such that the purely linear saturation $\bSd$ truncated in degree $d$ together with this natural transformation $\eta$ is a \nameft{Gabriel} monad of $\A$ w.r.t.\ $\C$.
\end{theorem}
Again, the statement of the theorem is valid for all $d \in \Z$.
The statement of the following immediate corollary and the proof the theorem assume $d \leq 0$.
\begin{coro} \label{coro:defect_of_saturation}
  The nonnegative integer $\max \{ \linregd C - d + 1, 0 \}$ is the precise count of recursion steps needed to achieve saturation.
\end{coro}

Thus, the linear regularity yields a better bound for computing zeroth cohomologies and saturation than the \nameft{Castelnuovo-Mumford} regularity does.
However, the data structure for $\Egrlind$ suggested in Remark~\ref{rmrk:data_structure_egrlind} still requires the \nameft{Castelnuovo-Mumford} regularity.

\begin{proof}[Proof of Theorem~\ref{thm:tate_monad}]
  First, we construct the natural transformation $\eta_C$ for the complex $(C,\mu)\in\Egrlind$.
  Consider the cochain-isomorphism $\eta_C: C^{\ge r}\to \bS^{\ge r}(C)$ induced by the identity for $r > \linregd C$.
  Assume that $\eta_C$ is lifted to a cochain morphism $C^{\ge \ell+1}\to \bS^{\ge \ell+1}(C)$.
  The universal property of the purely linear kernels implies a completion of the square
    \begin{center}
    \begin{tikzpicture}
      \coordinate (right) at (3,0);
      \coordinate (above) at (0,1.5);
      \coordinate (rightabove) at ($(right)+(above)$);
      
      \node (00)                 {$\bSd(C)^\ell$};
      \node (10) at (right)      {$\bSd(C)^{\ell+1}$};
      \node (01) at (above)      {$C^\ell$};
      \node (11) at (rightabove) {$C^{\ell+1}$};
      
      \draw[-stealth'] (11) -- node[right] {$\eta_C^{\ell+1}$} (10);
      \draw[-stealth',dashed] (01) -- node[right] {$\eta_C^\ell$} (00);
      \draw[-stealth'] (00) -- (10);
      \draw[-stealth'] (01) -- (11);
      
    \end{tikzpicture}
    \end{center}
  by a morphism $\eta_C^\ell: C^\ell\to\bSd(C)^\ell$.
  Iteratively, we get the cochain-morphism $\eta_C: C\to \bSd(C)$.

  Now, we use Theorem~\ref{thm:Csaturating} to show that $\bSd$ together with $\eta$ is a \nameft{Gabriel} monad.
  
  \begin{enumerate}
    \item[\ref{thm:Csaturating}.\eqref{thm:Csaturating:ker}] $\C \subset \ker \bSd$: \\
      As $\eta_C$ is an isomorphism in high cohomological degrees, its kernel is contained in $\C$.
    \item[\ref{thm:Csaturating}.\eqref{thm:Csaturating:im}] $\bSd(\A) \subset \Sat_\C(\A)$: \\
      $\bSd(C)$ has only trivial cohomologies above the regularity of $C$.
      Below the regularity we use Proposition~\ref{prop:saturation_and_purely_linear_kernels}.
    \item[\ref{thm:Csaturating}.\eqref{thm:Csaturating:exact}] $G := \cores_{\Sat_\C(\A)} \bSd$ is exact: \\
      As $\bSd$ is the identity on objects and morphism in high cohomological degree, applying it to a short exact sequence in $\A$ yields a new sequence with $\A$-defects, which are bounded by the maximum of the regularities of said short exact sequence.
      Thus, the $\A$-defects are contained in $\C$.
      In particular, this sequence is exact when considered in $\Sat_\A(\C)$.
    \item[\ref{thm:Csaturating}.\eqref{thm:Csaturating:idem}] $\eta\bSd = \bSd\eta$: \\
      Truncated at cohomological degree $\ell$ above the regularity this is clear, since both natural transformations are induced by the identity.
      For lower degrees, this follows from the uniqueness of the universal morphism $\psi$ in the definition of purely linear kernels.
    \item[\ref{thm:Csaturating}.\eqref{thm:Csaturating:eta}] $\eta\iota$ is a natural isomorphism: \\
      Let $C\in\A$ be $\C$-saturated.
      We need to show that $\eta_C$ is a cochain isomorphism.
      This is clear in high cohomolical degrees, as $\bSd$ is the identity on objects and morphism there.
      Assume that $\eta_C$ restricted to $C^{\ge \ell+1}\to \bS^{\ge \ell+1}(C)$ for some $\ell\in\Z$ is a cochain isomorphism.
      Then, the morphism $\eta_C^\ell$ from the completion of the square
      \begin{center}
      \begin{tikzpicture}
        \coordinate (right) at (3,0);
        \coordinate (above) at (0,1.5);        
        \coordinate (rightabove) at ($(right)+(above)$);

        \node (00)                 {$\bSd(C)^\ell$};
        \node (10) at (right)      {$\bSd(C)^{\ell+1}$};
        \node (01) at (above)      {$C^\ell$};
        \node (11) at (rightabove) {$C^{\ell+1}$};

        \draw[-stealth'] (11) -- node[right] {$\eta_C^{\ell+1}$} (10);
        \draw[-stealth',dashed] (01) -- node[right] {$\eta_C^\ell$} (00);
        \draw[-stealth'] (00) -- (10);
        \draw[-stealth'] (01) -- (11);

      \end{tikzpicture}
      \end{center}
      is an isomorphism, since both $C$ and $\bSd(C)$ are saturated and, by Proposition~\ref{prop:saturation_and_purely_linear_kernels} $C^\ell$ and $\bSd(C)^\ell$ are purely linear kernels of $\mu^{\ell+1}$ and the morphism in cohomological degree $\ell+1$ of $\bSd(C)$, respectively.
      The uniqueness of purely linear kernels implies that $\eta_C$ restricted to $C^{\ge \ell}\to \bS^{\ge \ell}(C)$ is a cochain isomorphism, and so is $\eta_C$ by induction.
    \qedhere
  \end{enumerate}
\end{proof}

We stress that the above functors $\MM$, $\MMd$, $\RR$, $\RRd$, and $\bSd$ are constructive functors between constructively \nameft{Abel}ian categories.
We furthermore note that computing the natural transformation $\eta$ is constructive.

\section{The \nameft{Gabriel} monad of the category of coherent sheaves}\label{sec:sheaves}

In this section we prove that for any $d \in \Z$ the quotient category $\Sfpgrmodd/\Sfpgrmodd^0$ is equivalent to the category $\Coh \PP^n_B$ and that the corresponding \nameft{Gabriel} monad computes the (truncated) module of twisted global sections.

\begin{prop}\label{prop:truncated_quotients_are_sheaves}
  $\Coh\PP^n_B\simeq\Sfpgrmodd/\Sfpgrmodd^0$ for all $d\in\Z$.
\end{prop}
\begin{proof}
  First, the definitions directly imply $\Sfpgrmod^0\cap \Sfpgrmodd=\Sfpgrmodd^0$.
  Second, a preimage of $M\in \Sfpgrmod/\Sfpgrmod^0$ under $\Sfpgrmodd\to\Sfpgrmod/\Sfpgrmod^0$ is given by $M_{\ge d}$, since $M\cong M_{\ge d}$ in $\Sfpgrmod/\Sfpgrmod^0$.
  Now the second isomorphism theorem for \nameft{Abel}ian categories \cite[Prop.~3.2]{BL_SerreQuotients} implies the equivalence
  \[
    \Sfpgrmodd/\Sfpgrmodd^0\simeq\Sfpgrmod/\Sfpgrmod^0 \mbox{.}
  \]
  The latter category is equivalent to $\Coh\PP^n_B$ by \cite[Coro.~4.2]{BL_SerreQuotients}.
\end{proof}

A graded $S$-modules $M$ is called quasi finitely generated if each truncation $M_{\geq d}$ is finitely generated.
We denote by $\Sqfgrmod \subset \Sgrmod$ the full subcategory of such modules.
The functor
\[
    H^0_\bullet:\Coh\PP^n_B\to\Sqfgrmod: \shF\mapsto\bigoplus_{p\in\Z}H^0(\PP^n_B,\shF(p))
\]
computing the module of twisted global sections is right adjoint to the sheafification functor $\Sh: \Sqfgrmod \to \Coh \PP^n_B, M \mapsto \widetilde{M}$.
This was proved by \nameft{Serre} in the absolute case \cite[59]{FAC} and later by \nameft{Grothendieck} for the relative case.

Denote by $\Sh_{\geq d}:\Sfpgrmodd\to\Coh\PP^n_B$ the restriction of $\Sh$ to $\Sfpgrmodd$ and by
\[
  H^0_{\ge d}:\Coh\PP^n_B\to\Sfpgrmodd: \shF\mapsto\bigoplus_{p \in \Z_{\geq d}}H^0(\PP^n_B,\shF(p))
\]
the functor computing the truncated module of twisted global sections.
It follows that $H^0_{\ge d}$ is the right adjoint of $\Sh_{\geq d}$.

\begin{prop} \label{prop:monads_computing_H0}
  The monad $H^0_{\ge d}(\widetilde{\,\cdot\,}) = H^0_{\ge d}\circ\Sh_{\geq d}$ is a \nameft{Gabriel} monad of $\Sfpgrmodd$ w.r.t.\ the localizing subcategory $\Sfpgrmod^0_{\geq d}$.
  In particular, any \nameft{Gabriel} monad computes the truncated module of twisted global sections.
\end{prop}
\begin{proof}
  Let $\QQ_{\ge d}:\Sfpgrmodd\to\Sfpgrmodd/\Sfpgrmodd^0$ be the canonical functor.
  The equivalence in Proposition~\ref{prop:truncated_quotients_are_sheaves} is constructed as a functor $\alpha_{\ge d}:\Sfpgrmodd/\Sfpgrmodd^0\to\Coh\PP^n_B$ with $\alpha_{\ge d}\circ\QQ_{\ge d}\simeq\Sh_{\geq d}$.
  An easy calculation shows that a right adjoint of $\QQ_{\ge d}$ is given by $\SS_{\ge d}:=H^0_{\ge d}\circ\alpha_{\ge d}$.
  In particular, $\SS_{\ge d}\circ\QQ_{\ge d}$ is a \nameft{Gabriel} monad of $\Sfpgrmodd$ w.r.t.\ $\Sfpgrmod^0_{\geq d}$ by \cite[Lemma~4.3]{BL_Monads}.
  Now the claim follows, as $\SS_{\ge d}\circ\QQ_{\ge d} = H^0_{\ge d}\circ\alpha_{\ge d}\circ\QQ_{\ge d}\simeq H^0_{\ge d}\circ\Sh_{\geq d} = H^0_{\ge d}(\widetilde{\,\cdot\,})$.
\end{proof}

\begin{coro} \label{coro:monads_computing_H0}
  There exist natural isomorphisms
  \[
    H^0_{\geq d}(\widetilde{M}) \cong D_{\m,\geq d}(M)
    \quad \mbox{and} \quad
    \RR\left( H^0_{\geq d}(\widetilde{M})\right ) \cong \bSd(\RR(M)) \mbox{,}
  \]
  in particular for $i \geq d$
  \[
    H^0(\widetilde{M}(i)) \cong \left(D_{\m,\geq d}(M)\right)_i \cong \left(\bSd(\RR(M))\right)^i_i \mbox{.}
  \]
\end{coro}

\begin{rmrk} \label{rmrk:E1}
  In the absolut case, i.e., when $B = k$ is a field, the (objects of the truncated) \nameft{Tate} resolution $\T^{\geq d}(M)$ relate to the higher cohomology modules $H^q_{\geq d}(\widetilde{M})$ by \cite{EFS}
  \begin{equation} \label{eq:Tate} \tag{*}
    \T^{\geq d}(M)^i = \bigoplus_{q=0}^{\min\{n,i - d\}} \omega_E \otimes_k H^q\left(\widetilde{M}(i-q)\right)\mbox{,}
  \end{equation}
  while the (truncated) purely linear saturation directly extracts $H^0$:
  \[
    \bSd(\RR(M))^i = \omega_E \otimes_B H^0\left(\widetilde{M}(i)\right) \mbox{.}
  \]
  
  In the relative case the analogue of \eqref{eq:Tate} is more subtle:
  The \nameft{Tate} resolution $\T^{\geq d}(M)$ is by its bi-graded structure in fact a multi-complex $\T^{\geq d,\bullet}(M)$.
  Since each multi-complex is a filtered complex and hence induces a spectral sequence\footnote{The vertical morphisms of this multi-complex are the morphisms between the graded summands represented by scalar matrices (i.e., degree zero in $E$).
  This differs from the \textsc{Macaulay2} convention used in \cite{ES}, where these morphisms are arranged ``diagonally up and to the right'' (cf.~\cite[Chapter~3]{ES}).
  Hence, we do not arrange the direct summands of the modules in the \nameft{Tate} resolution vertically, but diagonally up and to the left.}
  $E^{p,q}(\T^{\geq d,\bullet}(M)) \Longrightarrow 0$, where for each row-complex on the first page the following isomorphism holds
  \[
    E_1^{\geq d,q} \! \left(\T^{\geq d,\bullet}(M)\right) \cong \RR\left( H^q_{\geq d}(\widetilde{M})\right)\mbox{.}
  \]
  This is implicit in \cite{ES}, see also \cite[Corollary~3.6]{EFS}.
  Thus, the relation between the purely linear saturation and the \nameft{Tate} resolution is given by
  \[
    E_1^{\geq d,0} \! \left(\T^{\geq d,\bullet}(M)\right) \cong \bSd(\RR(M)) \cong \RR\left( H^0_{\geq d}(\widetilde{M})\right) \mbox{.}
  \]
  The first isomorphism is not a priori obvious in the relative case and gives a direct way to compute $\RR\left( H^0_{\geq d}(\widetilde{M})\right)$ via $\bSd(\RR(M))$ without computing (most of) the \nameft{Tate} resolution.\footnote{In absolute case $B=k$ the isomorphism easily follows from the fact that the bottom complex $E_1^{\geq d,0} \! \left(\T^{\geq d,\bullet}(M)\right)$ of the first spectral sequence is already the subcomplex of the \nameft{Tate} resolution consisting of those direct summands of the objects in $\T^{\geq d}(M)$ where the degree of the socle equals the cohomological degree.}
\end{rmrk}

\def\cprime{$'$} \def\cprime{$'$} \def\cprime{$'$} \def\cprime{$'$}
  \def\cprime{$'$}
\providecommand{\bysame}{\leavevmode\hbox to3em{\hrulefill}\thinspace}
\providecommand{\MR}{\relax\ifhmode\unskip\space\fi MR }
\providecommand{\MRhref}[2]{%
  \href{http://www.ams.org/mathscinet-getitem?mr=#1}{#2}
}
\providecommand{\href}[2]{#2}

\end{document}